\documentclass[10pt,psamsfonts]{amsart}
\usepackage{amssymb,accents}
\usepackage{amsmath}
\usepackage{graphicx}
\usepackage{amscd}
\usepackage{amsfonts}
\usepackage{amsbsy}
\usepackage[T1]{fontenc}
\usepackage[english]{babel}

\usepackage{enumerate}

\usepackage{centernot}
\usepackage{mathtools}

\usepackage{color}
\usepackage{hyperref}

\newtheorem{theorem}{Theorem}[section]
\newtheorem{prop}[theorem]{Proposition}
\newtheorem{lemma}[theorem]{Lemma}

\newtheorem*{problem*}{Problem}
\newtheorem{hypo}[theorem]{Main hypotheses}
\newtheorem*{hypo*}{Main hypotheses}

\newtheorem*{conjecture*}{Conjecture}

\newtheorem*{moran*}{Moran's Theorem (1946)}

\newtheorem*{matti*}{Mattila's Counterexample}

\newtheorem*{schief*}{Schief's Counterexample}

\theoremstyle{definition}
\newtheorem{definition}[theorem]{Definition}
\newtheorem{example}[theorem]{Example}

\theoremstyle{remark}
\newtheorem{remark}[theorem]{Remark}

\numberwithin{equation}{section}

\newcommand{\R}{\mathbb R}
\newcommand{\Z}{\mathbb Z}

\newcommand{\N}{\mathbb N}

\def \inf{{\rm inf}}

\def \dim{{\rm dim \,}}
\def \dm{{\rm diam \,}}

\def \cd{{\rm Card \,}}

\newcommand \F{\mathcal{F}}

\newcommand \St{{\rm St}}

\newcommand \ef{\mathbf{\Gamma}}
\newcommand \gf{\mathbf{\Delta}}

\newcommand \uno{\dim_{\ef}^{1}}
\newcommand \dos{\dim_{\ef}^{2}}
\newcommand \tres{\dim_{\ef}^{3}}

\newcommand \cuatro{\dim_{\ef}^{4}}
\newcommand \cinco{\dim_{\ef}^{5}}

\newcommand \h{\dim_{\rm H}}
\newcommand \bc{\dim_{\rm B}}

\newcommand \eps{\varepsilon}

\newcommand \anf{\mathcal{A}_n(F)}

\newcommand \hks{\mathcal{H}_{k}^s(F)}

\let\emptyset\varnothing

\begin{document}

\title[Calculating Hausdorff dimension in higher dimensional GF-spaces]
{Calculating Hausdorff dimension in higher dimensional spaces}

\author[M.A. S\'anchez-Granero and M. Fern\'andez-Mart\'{\i}nez]{M.A. S\'anchez-Granero$^1$ and M. Fern\'andez-Mart\'{\i}nez$^{2}$}

\address{$^{1}$ Departamento de Matem\'aticas, Universidad de Almer\'{\i}a, 04120 La Ca\~nada de San Urbano, Almer\'{\i}a (SPAIN)}

\email{misanche@ual.es}

\address{$^{2}$ University Centre of Defence at the Spanish Air Force Academy, MDE-UPCT, 30720 Santiago de la Ribera, Murcia (SPAIN)}
\email{fmm124@gmail.com}

\thanks{
\newline \indent 
Both authors are partially supported by grant No.~MTM2015-64373-P (MINE\-CO/FEDER, UE).
The first author also acknowledges the support of CDTIME and the second author also acknowledges the partial support of 
grant No.~19219/PI/14 from Fundaci\'on S\'eneca of Regi\'on de Murcia.
}


\begin{abstract}
In this paper, we prove the identity $\h(F)=d\cdot \h(\alpha^{-1}(F))$, where $\h$ denotes Hausdorff dimension, $F\subseteq \R^d$, and $\alpha:[0,1]\to [0,1]^d$ is a function whose constructive definition is addressed from the viewpoint of the powerful concept of a fractal structure. Such a result stands particularly from some other results stated in a more general setting. Thus, Hausdorff dimension of higher dimensional subsets can be calculated from Hausdorff dimension of $1-$dimensional subsets of $[0,1]$. 
As a consequence, Hausdorff dimension becomes available to deal with the effective calculation of the fractal dimension in applications by applying a procedure contributed by the authors in previous works.   
It is also worth pointing out that our results generalize both Skubalska-Rafaj{\l}owicz and Garc{\'{i}}a-Mora-Redtwitz theorems.
\end{abstract}

\maketitle

\section{Introduction}

In the mathematical literature there can be found (at least) a pair of theoretical results allowing the calculation of the box dimension of Euclidean objects in $\R^d$ in terms of the box dimension of $1-$dimensional Euclidean subsets. To attain such results, the concept of a \emph{space-filling curve} plays a key role. By a space-filling curve we shall understand a continuous map $\F$ from $I_1=[0,1]$ \emph{onto} the $d-$dimensional unit cube, $I_d=[0,1]^d$. It turns out that a one-to-one correspondence can be stated among closed real subintervals of the form $[k\, \delta^{nd},(1+k)\, \delta^{nd}]$ for $k=0,1,\ldots,\delta^{-nd}-1$, and sub-cubes $\F([k\, \delta^{nd},(1+k)\, \delta^{nd}])$ with lengths equal to $\delta^{n}$, where $\delta$ is a value depending on each space-filling curve. For instance, $\delta=\tfrac 12$ in both Hilbert's and Sierpi\'nski's square-filling curves, and $\delta=\tfrac 13$ in the case of the Peano's filling curve. It is worth pointing out that space-filling curves satisfy the two following properties.
\begin{remark}
Let $\F:I_1\to I_d$ be a space-filling curve. The two following hold.
\begin{enumerate}[(i)]
\item $\F$ is continuous and lies under the H\"{o}lder condition, i.e., $\|\F(x)-\F(y)\|\leq \kappa_d\cdot |x-y|^{\frac 1d}$ for all $x,y\in I_1$, where $\|\cdot\|$ denotes the Euclidean norm (induced in $I_d$), and $\kappa_d>0$ is a constant which depends on $d$.\label{prop:measure-preserving}
\item $\F$ is Lebesgue measure preserving, namely, $\mu_d(B)=\mu_1(\F^{-1}(B))$ for each Borel subset $B$ of $I_d$, where $\mu_d$ denotes the Lebesgue measure in $I_d$ and $\F^{-1}(B)=\{t\in I_1:\F(t)\in B\}$.\label{prop:Lebesgue-measure-preserving}
\end{enumerate}
\end{remark}
As it was stated in \cite[Subsection 3.1]{SkubalskaRafajowicz2005}, many space-filling curves satisfy (\ref{prop:measure-preserving}). On the other hand, despite $F$ cannot be invertible, it can be still proved that $F$ is a.e.~one-to-one (c.f.~\cite{SkubalskaRafajowicz2001,Milne1980}).
Following the above, Skubalska-Rafaj{\l}owicz stated the following result in 2005.
\begin{theorem}[c.f.~Theorem 1 in \cite{SkubalskaRafajowicz2005}]\label{teo:Skubalska-Rafajowicz2005}
Let $F$ be a subset of $I_d$ and assume that there exists $\bc(F)$. Then $\bc(\Psi(F))$ also exists and it holds that
\[
\bc(F)=d\cdot \bc(\Psi(F)),
\]
where $\Psi:I_d\to I_1$ is a quasi-inverse (in fact, a right inverse) of $\F$, namely, it satisfies that $\Psi(x)\in \F^{-1}(x)$, i.e., $\F(\Psi(x))=x$ for all $x\in I_d$.
\end{theorem}
The applicability of Theorem \ref{teo:Skubalska-Rafajowicz2005} for fractal dimension calculation purposes depends on a constructive method to properly define that quasi-inverse $\Psi$. In other words, for each $x\in I_d$, it has to be (explicitly) specified how to select a pre-image of $x$. Interestingly, for some Lebesgue measure preserving space-filling curves (including the Hilbert's, the Peano's, and the Sierpi\'nski's ones), it holds that $\F^{-1}(\{x\})$ is either a single point or a finite real subset. As such, suitable definitions of $\Psi$ can be provided in these cases. It is worth noting that whether both properties (\ref{prop:measure-preserving}) and (\ref{prop:Lebesgue-measure-preserving}) stand, then the quasi-inverse $\Psi$ becomes Lebesgue measure preserving, i.e., $\mu_d(\Psi^{-1}(A\cap \Psi(I_d)))=\mu_1(A\cap \Psi(I_d))$ for each Borel subset $A$ of $I_1$. Moreover, since the (Lebesgue) measure of $\Psi(B)\setminus F^{-1}(B)$ is equal to zero, then we have $\mu_1(\Psi(B))=\mu_1(\F^{-1}(B))=\mu_d(B)$ for all Borel subset $B$ of $I_d$. Therefore, if $\dm(\Psi(B))=\delta$, then $\dm(B)\leq \delta^{\frac 1d}$.

On the other hand, Garc\'{\i}a et al.~recently contributed a theoretical result also allowing the calculation of the box dimension of $d-$dimensional Euclidean subsets in terms of an asymptotic expression involving certain quantities to be calculated from $1-$dimensional subsets. To tackle with, they used the concept of a \emph{$\delta-$uniform curve}, that may be defined as follows. Let $\delta>0$. We recall that a $\delta-$cube in $\R^d$ is a set of the form $[k_1\, \delta,(1+k_1)\, \delta]\times \cdots\times [k_d\, \delta,(1+k_d)\, \delta]$, where $k_1,\ldots,k_d\in \Z$. Let $\mathcal{M}_{\delta}(I_d)$ denote the class of all $\delta-$cubes in~$I_d$. Thus, if $N\geq 1$ and $\delta=\frac 1N$, we shall understand that $\gamma:I_1\to I_d$ is a $\delta-$uniform curve in $I_d$ if there exist a $\delta-$cube in $I_d$, a $\delta^d-$cube in $I_1$, and a one-to-one correspondence, $\phi:\mathcal{M}_{\delta^d}(I_1)\to \mathcal{M}_{\delta}(I_d)$, such that $\gamma(J)\subset \phi(J)$ for all $J\in \mathcal{M}_{\delta^d}(I_1)$ (c.f.~\cite[Definition 3.1]{GARCIA2017}).
Moreover, let $s>0$, $F$ be a subset of $I_d$, and $\mathcal{N}_{\delta}(F)$ be the number of $\delta-$cubes in $I_d$ that intersect $F$. The $s-$body of $F$ is defined as $F_s=\{x\in I_d:\|x-y\|\leq s \text{ for some } y\in F\}$. Following the above, their main result is stated next.
\begin{theorem}[c.f.~Theorem 4.1 in \cite{GARCIA2017}]\label{teo:GARCIA2017}
Let $N>1$, $\delta=\frac 1N$, and $\gamma_{\delta}:I_1\to I_d$ be an injective $\delta-$uniform curve in $I_d$. Moreover, let $F$ be a (nonempty) subset of $I_d$, and $F_s$ its $s-$body, where $s=\delta\, \sqrt{d}$. Then the (lower/upper) box dimension of $F$ can be calculated throughout the next (lower/upper) limit:
\[
\bc(F)=\lim_{\delta\to 0}\frac{\log \mathcal{N}_{\delta^d}(\gamma_{\delta}^{-1}(F_s))}{-\log \delta}.
\]
\end{theorem}
It is worth mentioning that Theorem~\ref{teo:GARCIA2017} is supported by the existence of injective $\delta-$uniform curves in $I_d$ as the result below guarantees.
\begin{prop}[c.f.~Lemma 3.1 and Corollary 3.1~in \cite{GARCIA2017}]\label{lema:GARCIA2017}
Under the same hypotheses as in Theorem~\ref{teo:GARCIA2017}, the two following stand.
\begin{enumerate}
\item There exists an injective $\delta-$uniform curve in $I_d$, $\gamma_{\delta}:I_1\to I_d$.
\item $\log \mathcal{N}_{\delta^d}(\gamma_{\delta}^{-1}(F_s))=\log \mathcal{N}_{\delta}(F)+O(1)$.
\end{enumerate}
\end{prop}
From a novel viewpoint, along this article, we shall apply the powerful concept of a \emph{fractal structure} in order to extend both Theorems \ref{teo:Skubalska-Rafajowicz2005} and \ref{teo:GARCIA2017} to the case of Hausdorff dimension. Roughly speaking, a fractal structure is a countable family of coverings which throws more accurate approximations to the irregular nature of a given set as deeper stages within its structure are explored (c.f.~Subsection~\ref{sub:fs} for a rigorous description). In this paper, we shall contribute the following result in the Euclidean setting.
\begin{theorem}\label{teo:eu-dim-to-R-intro}
There exists a curve $\alpha:I_1\to I_d$ such that for each subset $F$ of $\R^d$, the two following hold:
\begin{enumerate}[(i)]
\item If there exists $\bc(F)$, then $\bc(\alpha^{-1}(F))$ also exists, and $\bc(F)=d\cdot \bc(\alpha^{-1}(F))$.
\item $\h(F)=d\cdot \h(\alpha^{-1}(F))$.
\end{enumerate}
\end{theorem}
As such, Theorem~\ref{teo:eu-dim-to-R-intro} gives the equality (up to a factor, namely, the embedding dimension) between the box dimension of a $d-$dimensional subset $F$ and the box dimension of its pre-image, $\alpha^{-1}(F)\subseteq \R$. Interestingly, such a theorem also allows calculating Hausdorff dimension of $d-$dimensional Euclidean subsets in terms of Hausdorff dimension of their $1-$dimensional pre-images via $\alpha$. It is worth pointing out that Section \ref{sec:how-to-construct} provides an approach allowing the construction of that map $\alpha:I_1\to I_d$ (as well as appropriate fractal structures) to effectively calculate the fractal dimension by means of Theorem \ref{teo:curves}. It is also worth noting that Theorem~\ref{teo:eu-dim-to-R-intro} stands as a consequence of some other results proved in more general settings (c.f.~Section~\ref{sec:4}). 

More generally, let $X,Y$ be a pair of sets. The main goal in this paper is to calculate the (more awkward) fractal dimension of objects contained in $Y$ in terms of the (easier to be calculated) fractal dimension of subsets of $X$  through an appropriate function $\alpha:X\to Y$. 
In other words, we shall guarantee the existence of a map $\alpha:X\to Y$ satisfying some desirable properties allowing to achieve the identity $\dim(F)=d\cdot \dim(\alpha^{-1}(F))$, where $F\subseteq Y, \alpha^{-1}(F)\subseteq X$, and $\dim$ refers to fractal dimensions I, II, III, IV, and V (introduced in previous works by the authors, c.f. \cite{DIM3,DIM1,DIM4}), as well as the classical fractal dimensions, namely, both box and Hausdorff dimensions. The nature of both spaces $X$ and $Y$ will be unveiled along each section in this paper. Interestingly, our results could be further applied to calculate the fractal dimension in non-Euclidean contexts including the domain of words (c.f.~\cite{FernandezMartinez2012}) and metric spaces such as the space of functions or the hyperspace of $Y$ (namely, the set containing all the closed or compact subsets of $Y$) to list a few. For $X$ we can use $[0,1]$, where calculations are easier, but also other spaces like the Cantor set $\{0,1\}^\N$ which is also a place where the calculation are easy.

The structure of this article is as follows. Firstly, Section \ref{sec:2} contains the basics on the fractal dimension models for a fractal structure that will support the main results to appear in upcoming sections. Section \ref{sec:3} is especially relevant since it provides the main requirements to be satisfied in most of the theoretical results contributed along this paper (c.f.~Main hypotheses \ref{hypo:main}). It is worth mentioning that such conditions are satisfied, in particular, by the natural fractal structure on each Euclidean subset (c.f.~Definition \ref{def:nfs}). Subsection~\ref{sub:bc-type} contains several results allowing the calculation of the box type dimensions (namely, fractal dimensions I, II, III, and standard box dimension, as well) for a map $\alpha:X\to Y$ and generic spaces $X$ and $Y$, each of them endowed with a fractal structure satisfying some conditions. Similarly, in Subsection~\ref{sub:h-type}, we explain how to deal with the calculation of Hausdorff type dimensions (i.e., fractal dimensions IV, V, and classical Hausdorff dimension). As a consequence of them, in Section~\ref{sec:4} we shall prove some results for both the box and Hausdorff dimensions. Also, we would like to highlight Theorem \ref{teo:eu-dim-to-R-intro} as a more operational version of both Theorems \ref{teo:bc-to-R} and \ref{teo:H-to-R} in the Euclidean setting (c.f.~Section~\ref{sec:5}). That result becomes especially appropriate to tackle with applications of fractal dimension in higher dimensional Euclidean spaces and lies in line with both Theorems~\ref{teo:Skubalska-Rafajowicz2005} and \ref{teo:GARCIA2017} (with regard to the box dimension). Finally, in Section~\ref{sec:6}, we explore a constructive approach to define an appropriate function $\alpha:X\to Y$ satisfying all the required conditions. For illustration purposes, we conclude the paper by some applications of that result to iteratively construct both the Hilbert's square-filling curve as well as a curve filling the whole Sierpi\'nski triangle. 

\section{Key concepts and starting results}\label{sec:2}

\subsection{Fractal structures}\label{sub:fs}
Fractal structures were first sketched by Bandt and Retta in \cite{bandt1992} and formally defined afterwards by Arenas and S\'anchez-Granero in \cite{SG99A} to characterize non-Archimedean quasi-metrization. 
By a covering of a nonempty set $X$, we shall understand a family $\Gamma$ of subsets of $X$ such that $X=\cup\{A:A\in \Gamma\}$. Let $\Gamma_1$ and $\Gamma_2$ be two coverings of $X$. The notation $\Gamma_2\prec \Gamma_1$ means that $\Gamma_2$ is a \emph{refinement} of $\Gamma_1$, i.e., for all $A\in \Gamma_2$, there exists $B\in \Gamma_1$ such that $A\subseteq B$. In addition, by $\Gamma_2\prec\prec \Gamma_1$, we shall understand both that $\Gamma_2\prec \Gamma_1$ and also that $B=\cup\{A\in \Gamma_2:A\subseteq B\}$ for all $B\in \Gamma_1$. Thus, a fractal structure on $X$ is a countable family of coverings $\ef=\{\Gamma_n\}_{n\in \N}$ such that $\Gamma_{n+1}\prec\prec \Gamma_n$. The pair $(X,\ef)$ is called a GF-space and covering $\Gamma_n$ is named \emph{level} $n$ of $\ef$. 
Along the sequel, we shall allow that a set could appear twice or more in any level of a fractal structure. 
Let $x\in X$ and $\ef$ be a fractal structure on $X$. Then we can define the star at $x$ in level $n\in \N$ as $\St(x,\Gamma_n)=\cup\{A\in \Gamma_n:x\in A\}$. Next, we shall describe the concept of natural fractal structure on any Euclidean space that will play a key role throughout this article.
\begin{definition}[c.f.~Definition 3.1 in \cite{DIM1}]\label{def:nfs}
The natural fractal structure on the Euclidean space $\R^d$ is given by the countable family of coverings $\ef=\{\Gamma_n:n\in \N\}$ with levels defined as
\[
\Gamma_n=\left\{\left[\frac{k_1}{2^n},\frac{1+k_1}{2^n}\right]\times \cdots \times \left[\frac{k_d}{2^n},\frac{1+k_d}{2^n}\right]:k_1,\ldots,k_d\in \Z\right\}.
\]
\end{definition}
As such, the natural fractal structure on $\R^d$ is just a tiling consisting of $\frac{1}{2^n}-$cubes on $\R^d$. Notice also that natural fractal structures may be induced on Euclidean subsets of $\R^d$. For instance, the natural fractal structure on $[0,1]\subset \R$ is the countable family of coverings $\ef$ with levels given by $\Gamma_n=\{[\frac{k}{2^n},\frac{1+k}{2^n}]:k=0,1,\ldots,2^n-1\}$ for all $n\in \N$.

\subsection{Fractal dimensions for fractal structures}\label{sub:fdims}
The fractal dimension models for a fractal structure involved along this paper, namely, fractal dimensions I, II, III, IV, and V, were introduced previously by the authors (c.f.~\cite{DIM3,DIM1,DIM4}) and proved to generalize both box and Hausdorff dimensions in the Euclidean setting (c.f.~\cite[Theorem 3.5, Theorem 4.7]{DIM1},\cite[Theorem 4.15]{DIM3}, \cite[Theorem 3.13]{DIM4}) through their natural fractal structures (c.f.~\cite[Definition 3.1]{DIM1}). Thus, they become ideal candidates to explore the fractal nature of subsets. Next, we recall the definitions of all the box type dimensions appeared along this article. 
\begin{definition}[box type dimensions]\label{def:fdims}
Let $F$ be a subset of $X$.
\begin{enumerate}[(1)]
\item (\cite{Pontrjagin1932}) If $X=\R^d$, then the (lower/upper) box dimension of $F$ is defined through the (lower/upper) limit
\[
\bc(F)=\lim_{\delta\to 0}\frac{\log \mathcal{N}_{\delta}(F)}{-\log \delta},
\] 
where $\mathcal{N}_{\delta}(F)$ can be calculated as the number of $\delta-$cubes that intersect $F$ (among other equivalent quantities).
\item (c.f.~Definition 3.3 in \cite{DIM1}) Let $\ef$ be a fractal structure on $X$. We shall denote $\anf=\{A\in \Gamma_n:A\cap F\neq \emptyset\}$ and $\mathcal{N}_n(F)=\cd(\anf)$, as well. The (lower/upper) fractal dimension I of $F$ is given by the next (lower/upper) limit:
\[
\uno(F)=\lim_{n\to \infty}\frac{\log \mathcal{N}_n(F)}{n\log 2}.
\]
\item Let $\ef$ be a fractal structure on a metric space $(X,\rho)$. 
\begin{enumerate}[(i)]
\item (c.f.~Definition 4.2 in \cite{DIM1}) Let us denote $\dm(F,\Gamma_n)=\sup\{\dm(A):A\in \anf\}$, where $\dm(A)=\sup\{\rho(x,y):x,y\in A\}$, as usual. The (lower/upper) fractal dimension II of $F$ is defined as 
\[
\dos(F)=\lim_{n\to \infty}\frac{\log \mathcal{N}_n(F)}{-\log \dm(F,\Gamma_n)}.
\] 
\item (c.f.~Definition 4.2 in \cite{DIM3}) Let $s>0$, assume that $\dm(F,\Gamma_n)\to 0$, and define
\[
\mathcal{H}_{n,3}^s(F)=\inf\left\{\sum \dm(A_i)^s:\{A_i\}_{i\in I}\in \mathcal{A}_{n,3}(F)\right\},
\]
where
$\mathcal{A}_{n,3}(F)=\{\mathcal{A}_l(F):l\geq n\}$.\label{dim:3}
Further, let $\hks=\lim_{n\to \infty}\mathcal{H}_{n,k}^s(F)$. The fractal dimension III of $F$ is the (unique) critical point satisfying that
\[
\tres(F)=\sup\{s\geq 0:\mathcal{H}_3^s(F)=\infty\}=\inf\{s\geq 0:\mathcal{H}_3^s(F)=0\}.
\]
\end{enumerate}
\end{enumerate}
\end{definition}
Let $(X,\rho)$ be a metric space, $\delta>0$, and $F$ be a subset of $X$. By a $\delta-$cover of $F$, we shall understand a countable family of subsets of $X$, $\{U_i\}_{i\in I}$, with $\dm(U_i)\leq \delta$ for all $i\in I$ and such that $F\subseteq \cup_{i\in I}U_i$. Next, we provide the definitions for all Hausdorff type definitions involved in this paper.
\begin{definition}[Hausdorff type dimensions]\label{def:fdims}
Let $(X,\rho)$ be a metric space, $s>0$, and $F$ be a subset of $X$.
\begin{enumerate}[(1)]
\item (\cite{MR1511917}) Let $\mathcal{C}_{\delta}(F)$ denote the class of all $\delta-$covers of $F$, define
\[
\mathcal{H}_{\delta}^s(F)=\inf\left\{\sum_{i\in I}\dm(U_i)^s:\{U_i\}_{i\in I}\in \mathcal{C}_{\delta}(F)\right\},
\]
and let the $s-$dimensional Hausdorff measure of $F$ be given by 
\[
\mathcal{H}_{\textrm{H}}^s(F)=\lim_{\delta\to 0}\mathcal{H}_{\delta}^s(F).
\] 
Hausdorff dimension of $F$ is the (unique) critical point satisfying that
\[
\h(F)=\sup\{s\geq 0:\mathcal{H}_H^s(F)=\infty\}=\inf\{s\geq 0:\mathcal{H}_H^s(F)=0\}.
\]
\item Let $\ef$ be a fractal structure on a metric space $(X,\rho)$, assume that $\dm(F,\Gamma_n)\to 0$, and define (c.f.~Definition 3.2 in \cite{DIM4})
\begin{enumerate}[(i)]
\item 
\[
\mathcal{H}_{n,4}^s(F)=\inf\left\{\sum \dm(A_i)^s:\{A_i\}_{i\in I}\in \mathcal{A}_{n,4}(F)\right\},
\]
where
$\mathcal{A}_{n,4}(F)=\{\{A_i\}_{i\in I}:A_i\in \cup_{l\geq n}\Gamma_l,F\subseteq \cup_{i\in I}A_i,\cd(I)<\infty\}$,\label{dim:4}
and $\mathcal{H}_4^s(F)=\lim_{n\to \infty}\mathcal{H}_{n,4}^s(F)$. The fractal dimension IV of $F$ is the (unique) critical point satisfying that
\[
\cuatro(F)=\sup\{s\geq 0:\mathcal{H}_4^s(F)=\infty\}=\inf\{s\geq 0:\mathcal{H}_4^s(F)=0\}.
\]
\item \[
\mathcal{H}_{n,5}^s(F)=\inf\left\{\sum \dm(A_i)^s:\{A_i\}_{i\in I}\in \mathcal{A}_{n,5}(F)\right\},
\]
where
$\mathcal{A}_{n,5}(F)=\{\{A_i\}_{i\in I}:A_i\in \cup_{l\geq n}\Gamma_l,F\subseteq \cup_{i\in I}A_i\}$,\label{dim:5}
and $\mathcal{H}_5^s(F)=\lim_{n\to \infty}\mathcal{H}_{n,5}^s(F)$. The fractal dimension V of $F$ is the (unique) critical point satisfying that
\[
\cinco(F)=\sup\{s\geq 0:\mathcal{H}_5^s(F)=\infty\}=\inf\{s\geq 0:\mathcal{H}_5^s(F)=0\}.
\]
\end{enumerate}
\end{enumerate}
\end{definition}
It is worth pointing out that fractal dimensions III, IV, and V always exist since the sequences $\{\mathcal{H}_{n,k}^s(F)\}_{n\in \N}$ are monotonic in $n\in \N$ for $k=3,4,5$. 

\subsection{Connections among fractal dimensions}

Next, we collect some theoretical links among the box (resp., Hausdorff) dimension and the fractal dimension models for a fractal structure introduced in previous Subsection \ref{sub:fdims}. The following result stands in the Euclidean setting. 

\begin{theorem}\label{teo:previo}
Let $\ef$ be the natural fractal structure induced on $F\subseteq \R^d$. The following statements hold.
\begin{enumerate}[(i)]
\item (c.f.~\cite{DIM1}, Theorem 3.5)\label{teo:bc=uno} $\bc(F)=\uno(F)$.
\item (c.f.~\cite{DIM1}, Theorem 4.7)\label{teo:bc=dos} $\bc(F)=\dos(F)$.
\item (c.f.~\cite{DIM3}, Theorem 4.15)\label{teo:bc=tres} $\bc(F)=\tres(F)$.
\item (c.f.~\cite{DIM4}, Theorem 3.12)\label{teo:4=H} $\h(F)=\cuatro(F)$ for each compact subset $F$ of\, $\R^d$.\label{teo:previo2}
\item (c.f.~\cite{DIM4}, Theorem 3.10)\label{teo:5=H} $\h(F)=\cinco(F)$.
\end{enumerate}
\end{theorem}
It is also worth pointing out that under the $\kappa-$condition for a fractal structure we recall next, the box dimension equals both fractal dimensions II and III on a generic GF-space.

\begin{definition}\label{def:k-cond}
Let $\ef$ be a fractal structure on $X$. We say that $\ef$ lies under the $\kappa-$condition if there exists a natural number $\kappa$ such that for all $n\in \N$, every subset $A$ of $X$ with $\dm(A)\leq \dm(\Gamma_n)$ intersects at most to $\kappa$ elements in $\Gamma_n$.
\end{definition}

\begin{theorem}[c.f.~\cite{DIM1}, Theorem 4.13 (1)]\label{teo:k-cond->bc=dos}
Assume that $\ef$ satisfies the $\kappa-$condition. If $\dm(F,\Gamma_n)\to 0$ and there exists $\bc(F)$, then $\bc(F)=\dos(F)$.
\end{theorem}

\begin{theorem}[c.f.~\cite{DIM3}, Theorem 4.17]\label{teo:k-cond->bc=tres}
Assume that $\ef$ is under the $\kappa-$condition. If $\dm(A)=\dm(F,\Gamma_n)$ for all $A\in \anf$, then $\bc(F)=\tres(F)$.
\end{theorem}

\section{Calculating the fractal dimension in higher dimensional spaces}\label{sec:3}

First, we would like to point out that all the results contributed along this section stand in the setting of metric spaces, whereas the results provided in both \cite{GARCIA2017} and \cite{SkubalskaRafajowicz2005} hold for Euclidean subsets regarding the box dimension.

Let $X$ and $Y$ denote metric spaces. The following hypothesis will be required in most of the theoretical results contributed hereafter.
\begin{hypo}\label{hypo:main}
Let $\alpha:X\to Y$ be a function between a pair of GF-spaces, $(X,\ef)$ and $(Y,\gf)$, with 
$\gf=\alpha(\ef)$. Assume, in addition, that
there exists a pair of real numbers, $d$ and $c\neq 0$, such that the following identity stands for each $A\in \Gamma_n$ and all $n\in \N$:
\begin{equation}\label{hypo:diam}
\dm(\alpha(A))^d=c\cdot \dm(A).	
\end{equation}
\end{hypo}


\subsection{Calculating the box type dimensions in higher dimensional spaces}\label{sub:bc-type}

\begin{lemma}\label{lema:1}
Let $F\subseteq Y$, $n\in \N$, and $A\in \Gamma_n$. Then 
\[
A\cap \alpha^{-1}(F)\neq \emptyset\Leftrightarrow \alpha(A)\cap F\neq~\emptyset.
\]
\end{lemma}

\begin{proof}~
Next, we shall prove both implications.
\begin{enumerate}
\item [($\Rightarrow$)] Let $x\in A\cap \alpha^{-1}(F)$. Thus, $\alpha(x)\in \alpha(A)$ as well as $\alpha(x)\in F$. Hence, $\alpha(x)\in \alpha(A)\cap F$, so $\alpha(A)\cap F\neq \emptyset$.
\item [($\Leftarrow$)] Let $y\in \alpha(A)\cap F$. Since $y\in \alpha(A)$, then there exists $a\in A$ such that $y=\alpha(a)$. Also, it holds that $a\in \alpha^{-1}(F)$ since $y=\alpha(a)\in F$. Hence, $a\in A\cap \alpha^{-1}(F)$, so $A\cap \alpha^{-1}(F)\neq \emptyset$. 
\end{enumerate}
\end{proof}


Let us consider the next two families of elements in levels $n$ of both $\ef$ and $\gf$:
\begin{align*}
\mathcal{A}_{\Gamma_n}(\alpha^{-1}(F))&=\{A\in \Gamma_n:A\cap \alpha^{-1}(F)\neq \emptyset\}\\
\mathcal{A}_{\Delta_n}(F)&=\{B\in \Delta_n:B\cap F\neq \emptyset\}.
\end{align*}
Additionally, we shall denote $\mathcal{N}_{\Gamma_n}(\alpha^{-1}(F))=\cd(\mathcal{A}_{\Gamma_n}(\alpha^{-1}(F)))$ and $\mathcal{N}_{\Delta_n}(F)=\cd(\mathcal{A}_{\Delta_n}(F))$, as well. It is worth pointing out that Lemma \ref{lema:1} yields the next result.
\begin{prop}\label{prop:num1}
Let $\alpha:X\to Y$ be a function between a pair of GF-spaces, $(X,\ef)$ and $(Y,\gf)$, with $\gf=\alpha(\ef)$, and $F\subseteq Y$. Then for each $n\in \N$, it holds that
\[
\mathcal{N}_{\Gamma_n}(\alpha^{-1}(F))=\mathcal{N}_{\Delta_n}(F).
\]
\end{prop}
As a consequence from Proposition~\ref{prop:num1}, the calculation of the fractal dimension I of $F\subseteq Y$ can be dealt with in terms of the fractal dimension I of its pre-image $\alpha^{-1}(F)\subseteq X$ via $\alpha$ as the following result highlights.

\begin{theorem}\label{teo:dim1-to-R}
Let $\alpha:X\to Y$ be a function between a pair of GF-spaces, $(X,\ef)$ and $(Y,\gf)$, with $\gf=\alpha(\ef)$, and $F\subseteq Y$. Then the (lower/upper) fractal dimension I of $F$ (calculated with respect to $\gf$) equals the (lower/upper) fractal dimension I of $\alpha^{-1}(F)$ (calculated with respect to $\ef$). In particular, if $\dim_{\gf}^1(F)$ exists, then $\dim_{\ef}^1(\alpha^{-1}(F))$ also exists (and reciprocally), and it holds that 
\[
\dim_{\gf}^1(F)=\uno(\alpha^{-1}(F)).
\] 
\end{theorem}
Interestingly, a first connection between the box dimension of $F\subseteq Y$ and the fractal dimension I of its pre-image via $\alpha$, $\alpha^{-1}(F)\subseteq X$, can be stated in the Euclidean setting.
\begin{theorem}\label{teo:bc=uno-to-R}
Let $F\subseteq [0,1]^d$, $\gf$ the natural fractal structure on $F$, and $\alpha:X\to [0,1]^d$ a function between the GF-spaces $(X,\ef)$ and $([0,1]^d,\gf)$, where $\gf=\alpha(\ef)$. Then the (lower/upper) box dimension of $F$ equals the (lower/upper) fractal dimension I of $\alpha^{-1}(F)$ (calculated with respect to $\ef$). In particular, if $\bc(F)$ exists, then $\dim_{\ef}^{1}(\alpha^{-1}(F))$ also exists (and reciprocally), and it holds that 
\[
\bc(F)=\dim_{\ef}^{1}(\alpha^{-1}(F)).
\] 
\end{theorem}

\begin{proof}
First, we have $\bc(F)=\dim_{\gf}^{1}(F)$, since $\gf$ is the natural fractal structure on $F\subseteq [0,1]^d$ (c.f.~Theorem \ref{teo:previo} (\ref{teo:bc=uno})). Thus, just apply Theorem \ref{teo:dim1-to-R} to get the result.
\end{proof}
Similarly to Theorem \ref{teo:dim1-to-R}, the following result stands for fractal dimension II.
\begin{theorem}\label{teo:dim2-to-R}
Let $F\subseteq Y$.
Under Main hypotheses \ref{hypo:main}, it holds that the (lower/upper) fractal dimension II of $F$ (calculated with respect to $\gf$) equals the (lower/upper) fractal dimension II of $\alpha^{-1}(F)$ (calculated with respect to $\ef$) multiplied by $d$. In particular, if $\dim_{\gf}^2(F)$ exists, then $\dim_{\ef}^{2}(\alpha^{1}(F))$ also exists (and reciprocally), and it holds that
\[
\dim_{\gf}^2(F)=d\cdot \dim_{\ef}^{2}(\alpha^{-1}(F)).
\]
\end{theorem}

\begin{proof}
First of all, for all $A\in \Gamma_n$, it holds that $c\cdot \dm(A)=\dm(\alpha(A))^d$ for some $c\neq 0$ and $d\in \R$ (c.f. Eq.~(\ref{hypo:diam})). Hence,
\begin{align*}
c\cdot \dm(\alpha^{-1}(F),\Gamma_n)&=c\cdot \sup\{\dm(A):A\in \Gamma_n,A\cap \alpha^{-1}(F)\neq \emptyset\}\\
&=\sup\{\dm(\alpha(A))^d:\alpha(A)\in \Delta_n,\alpha(A)\cap F\neq \emptyset\}\\
&=\dm(F,\Delta_n)^d \text{ for all } n\in \N.
\end{align*}
Thus, it holds that
\begin{equation}\label{eq:den2}
\dm(\alpha^{-1}(F),\Gamma_n)=\frac{1}{c}\cdot \dm(F,\Delta_n)^d \text{ for all } n\in \N.
\end{equation}
Accordingly,
\begin{align*}
\lim_{n\to \infty}\frac{\log \mathcal{N}_{\Gamma_n}(\alpha^{-1}(F))}{-\log \dm(\alpha^{-1}(F),\Gamma_n)}&=\lim_{n\to \infty}\frac{\log \mathcal{N}_{\Delta_n}(F)}{-\log \frac 1c\, \dm(F,\Delta_n)^d}\\
&=\frac 1d\cdot\lim_{n\to \infty}\frac{\log \mathcal{N}_{\Delta_n}(F)}{-\log \dm(F,\Delta_n)},
\end{align*}
where $\lim$ refers to the corresponding lower/upper limit. Notice also that both Eq.~(\ref{eq:den2}) and Proposition \ref{prop:num1} have been applied to deal with the second equality.
\end{proof}
Additionally, the following result for fractal dimension II stands similarly to Theorem~\ref{teo:bc=uno-to-R}.
\begin{theorem}\label{teo:bc=dos-to-R}
Let $F\subseteq [0,1]^d$, $\gf$ the natural fractal structure on $F$, and $\alpha:X\to [0,1]^d$ a function between the GF-spaces $(X,\ef)$ and $([0,1]^d,\gf)$, where $\gf=\alpha(\ef)$. 
Under Main hypotheses \ref{hypo:main}, the (lower/upper) box dimension of $F$ equals the (lower/upper) fractal dimension II of $\alpha^{-1}(F)$ (calculated with respect to $\ef$). In particular, if $\bc(F)$ exists, then $\dim_{\ef}^{2}(\alpha^{-1}(F))$ also exists (and reciprocally), and it holds that 
\[
\bc(F)=d\cdot \dim_{\ef}^{2}(\alpha^{-1}(F)).
\] 
\end{theorem}

\begin{proof}
The result follows immediately since
\[
\bc(F)=\dim_{\gf}^2(F)=d\cdot \dim_{\ef}^2(\alpha^{-1}(F)),
\]
where the first identity holds since $\gf$ is the natural fractal structure on $F$ (c.f.~Theorem \ref{teo:previo} (\ref{teo:bc=dos})) and the second equality stands by previous Theorem \ref{teo:dim2-to-R}.
\end{proof}
According to the previous result, the box dimension of $F\subseteq [0,1]^d$ may be calculated by the fractal dimension II of $\alpha^{-1}(F)\subseteq [0,1]$ (calculated with respect to $\ef$).
As such, the following result stands in the Euclidean setting as a consequence of Theorem \ref{teo:bc=dos-to-R}.

\begin{theorem}\label{teo:eu-bc-to-R}
Let $F\subseteq [0,1]^d$ and $\alpha:[0,1]\to [0,1]^d$ be a function between the GF-spaces $([0,1],\ef)$, where $\Gamma_n=\{[k\cdot 2^{-nd},(1+k)\cdot 2^{-nd}]:k=0,1,\ldots,2^{nd}-1\}$ are the levels of $\ef$, and $([0,1]^d,\gf)$, where $\gf$ is the natural fractal structure on $[0,1]^d$ and such that $\alpha(\ef)=\gf$. It holds that the (lower/upper) box dimension of $F$ equals the (lower/upper) box dimension of $\alpha^{-1}(F)$. In particular, if $\bc(F)$ exists, then $\bc(\alpha^{-1}(F))$ also exists (and reciprocally), and it holds that
\[
\bc(F)=d\cdot \bc(\alpha^{-1}(F)).
\]
\end{theorem}

\begin{proof}
Note that Main hypotheses \ref{hypo:main} is satisfied since $\alpha(\ef)=\gf$ and $\dm(A)=2^{-nd}$ and $\dm(\alpha(A))=2^{-n}\cdot \sqrt{d}$ for all $A\in \Gamma_n$ and $n\in \N$. In fact, we can take $c=d^{d/2}$ with $d$ being the embedding dimension.  
Hence, we have $\bc(F)=d\cdot \dos(\alpha^{-1}(F))$ for all $F\subseteq [0,1]^d$ due to Theorem \ref{teo:bc=dos-to-R}. 
In addition, it is worth pointing out that
\begin{itemize}
\item $\ef$ satisfies the $\kappa-$condition for $\kappa=3$.
\item $\dm(F,\Gamma_n)\to 0$ since $\dm(A)=2^{-nd}$ for all $A\in \Gamma_n$ and $n\in \N$.
\end{itemize}
Hence, Theorem \ref{teo:k-cond->bc=dos} gives $\dos(\alpha^{-1}(F))=\bc(\alpha^{-1}(F))$.
\end{proof}
Next step is to prove a similar result to both Theorems \ref{teo:dim1-to-R} and \ref{teo:dim2-to-R} for fractal dimension III. Firstly, we have the following
\begin{prop}
Under Main hypotheses~\ref{hypo:main}, the next identity stands:
\begin{equation}\label{eq:10}
\mathcal{H}_3^s(\alpha^{-1}(F))=\frac{1}{c^s}\cdot \mathcal{H}_3^{ds}(F)\text{ for all } s\geq 0.
\end{equation} 
\end{prop}

\begin{proof}
The Main hypotheses~\ref{hypo:main} give that $\dm(\alpha(A))^{d}=c\cdot \dm(A)$ for some $d$ and $c\neq 0$. Thus,
\begin{equation}\label{eq:dms}
\dm(A)^s=\frac{1}{c^s}\cdot \dm(\alpha(A))^{ds}.
\end{equation}
Hence, for all $s\geq 0$, we have
\begin{align*}
\sum\{\dm(A)^s:A\in \Gamma_n, A\cap \alpha^{-1}(F)\neq \emptyset\}  = \\
\frac{1}{c^s}\cdot \sum\{\dm(\alpha(A))^{ds}:A\in \Gamma_n, \alpha(A)\cap F\neq \emptyset\},
\end{align*}
where the equality is due to Eq.~(\ref{eq:dms}) and also by applying Lemma~\ref{lema:1}. Therefore $\mathcal{H}_{n,3}^s(\alpha^{-1}(F))=\frac{1}{c^s}\cdot \mathcal{H}_{n,3}^{ds}(F)$. The result follows by letting $n\to \infty$.
\end{proof}
Hence, we have the expected
\begin{theorem}\label{teo:dim3-to-R}
Let $F\subseteq Y$. Under Main hypotheses~\ref{hypo:main}, it holds that 
\[
\dim_{\gf}^3(F)=d\cdot \dim_{\ef}^3(\alpha^{-1}(F)).
\] 
\end{theorem}

\begin{proof}
Firstly, by Eq.~(\ref{eq:10}), it holds that 
\[
\mathcal{H}_3^s(\alpha^{-1}(F))=\frac{1}{c^s}\cdot \mathcal{H}_3^{ds}(F)\text{ for all } s\geq 0.
\]
Thus, $\mathcal{H}_3^s(\alpha^{-1}(F))=0$ implies $\mathcal{H}_3^{ds}(F)=0$. Therefore, $s\geq \frac 1d\cdot \dim_{\gf}^3(F)$ for all $s\geq \dim_{\ef}^3(\alpha^{-1}(F))$. In particular, we have
\begin{equation}\label{eq:2}
\dim_{\ef}^3(\alpha^{-1}(F))\geq \frac 1d\cdot \dim_{\gf}^3(F).	
\end{equation}
Conversely, $\mathcal{H}_3^{ds}(F)=0$ leads to $\mathcal{H}_3^s(\alpha^{-1}(F))=0$, also by Eq.~(\ref{eq:10}). Thus, $s\geq \dim_{\ef}^3(\alpha^{-1}(F))$ for all $s\geq \frac 1d\cdot \dim_{\gf}^3(F)$. Hence,
\begin{equation}\label{eq:3}
\frac 1d\cdot \dim_{\gf}^3(F)\geq \dim_{\ef}^3(\alpha^{-1}(F)). 	
\end{equation} 
The result follows due to both Eqs.~(\ref{eq:2}) and (\ref{eq:3}). 
\end{proof}
The following result regarding fractal dimension III stands similarly to Theorem~\ref{teo:bc=dos-to-R}.
\begin{theorem}\label{teo:bc-to-R-via-dim3}
Let $F\subseteq [0,1]^d$, $\gf$ be the natural fractal structure on $[0,1]^d$, and $\alpha:X\to [0,1]^d$ a function between the GF-spaces $(X,\ef)$ and $([0,1]^d,\gf)$ with $\gf=\alpha(\ef)$. 
Under Main hypotheses \ref{hypo:main}, if $\bc(F)$ exists, it holds that 
\[
\bc(F)=d\cdot \tres(\alpha^{-1}(F)).
\] 
\end{theorem}

\begin{proof}
In fact, we have $\bc(F)=\dim_{\gf}^3(F)$ since $\gf$ is the natural fractal structure on $F$ (c.f.~Theorem \ref{teo:previo}.\ref{teo:bc=tres}). Finally, Theorem \ref{teo:dim3-to-R} gives the result.
\end{proof}
It is worth mentioning that Theorem \ref{teo:bc-to-R-via-dim3} implies that the box dimension of $F\subseteq [0,1]^d$ can be calculated throughout the fractal dimension III of $\alpha^{-1}(F)\subseteq [0,1]$ (calculated with respect to $\ef$).

\subsection{Calculating Hausdorff type dimensions in higher dimensional spaces}\label{sub:h-type}

Similarly to Lemma~\ref{lema:1}, the next implication stands.
\begin{lemma}\label{lema:2}
Let $\{A_i\}_{i\in I}$ be a collection of elements of $\ef$, and $F\subseteq Y$. Then
\[
\alpha^{-1}(F)\subseteq \cup_{i\in I}A_i\Rightarrow F\subseteq \cup_{i\in I}\alpha(A_i).
\]
\end{lemma}

\begin{proof}
If $x\in F$, then let $y\in \alpha^{-1}(x)$ be such that $x=\alpha(y)$. Since $y\in \alpha^{-1}(F)\subseteq \cup_{i\in I}A_i$, then there exists $j\in I$ such that $y\in A_j$. Hence, $x=\alpha(y)\in \alpha(A_j):j\in I$, so $x\in \cup_{i\in I}\alpha(A_i)$. 
\end{proof}

\begin{prop}\label{prop:ds<s}
Under Main hypotheses~\ref{hypo:main}, the next inequality holds:
\begin{equation}\label{eq:1}
\frac{1}{c^s}\cdot \mathcal{H}_5^{ds}(F)\leq \mathcal{H}_5^s(\alpha^{-1}(F))\text{ for all } s\geq 0.
\end{equation} 
\end{prop}

\begin{proof}
Let $s\geq 0$ and $n\in \N$.
First, it holds that $\dm(A)^s=\frac{1}{c^s}\cdot \dm(\alpha(A))^{ds}$ for each $A$ in some level $\geq n$ of $\ef$ (c.f.~Eq.~(\ref{hypo:diam})). Hence,
\begin{align*}
\mathcal{H}_{n,5}^{s}(\alpha^{-1}(F))&=\inf\left\{\sum_{i\in I}\dm(A_i)^s:A_i\in \Gamma_m:m\geq n, \alpha^{-1}(F)\subseteq \cup_{i\in I}A_i\right\}\\
&\geq\frac{1}{c^s}\cdot \inf\left\{\sum_{i\in I}\dm(\alpha(A_i))^{ds}:A_i\in \Gamma_m, m\geq n,F\subseteq \cup_{i\in I}\alpha(A_i)\right\}\\
&=\frac{1}{c^s}\cdot \mathcal{H}_{n,5}^{ds}(F).
\end{align*}
It is worth mentioning that Lemma~\ref{lema:2} has been applied in the inequality above. 
Letting $n\to \infty$, the result follows.
\end{proof}

\begin{theorem}\label{cor:<}
Under Main hypotheses~\ref{hypo:main}, it holds that 
\[
\dim_{\gf}^5(F)\leq d\cdot \dim_{\ef}^5(\alpha^{-1}(F)).
\]
\end{theorem}

\begin{proof}
Notice that $\mathcal{H}_5^{ds}(F)=0$ for all $s\geq 0$ such that $\mathcal{H}_5^s(\alpha^{-1}(F))=0$ (c.f.~Eq.~(\ref{eq:1})). Thus, $s\geq \frac 1d\cdot \dim_{\gf}^5(F)$ for all $s> \dim_{\ef}^5(\alpha^{-1}(F))$. It follows that $\dim_{\gf}^5(F)\leq d\cdot \dim_{\ef}^{5}(\alpha^{-1}(F))$.
\end{proof}
However, a reciprocal for Theorem~\ref{cor:<} becomes more awkward. To tackle with, let us introduce the following concept.
\begin{definition}\label{def:conds}
Let $\ef$ be a fractal structure on $X$. We shall understand that 
$\ef$ satisfies the finitely splitting property if there exists $\kappa'\in \N$ such that $\cd(\{A\in \Gamma_{n+1}:A\subseteq B\})\leq \kappa'$ for all $B\in \Gamma_n$ and all $n\in \N$.\label{def:k'-cond}
\end{definition}

\begin{prop}\label{prop:awk}
Let $F\subseteq Y$ and $s\geq 0$. Assume that $\gf$ is finitely splitting and satisfies the $\kappa-$condition. Under Main hypotheses~\ref{hypo:main}, it holds that $\mathcal{H}_5^{ds}(F)=0$ implies that $\mathcal{H}_5^s(\alpha^{-1}(F))=0$.
\end{prop}

\begin{proof}
Let $s\geq 0$ be such that $\mathcal{H}_5^{ds}(F)=0$. By Main hypothesis~\ref{hypo:main}, there exists $c\neq 0$ and $d\in \R$ such that $\dm(\alpha(A))^d=c\cdot \dm(A)$ for all $A\in \Gamma_n$ and all $n\in \N$. 
Moreover, let $\eps>0$ and $\eps'=c^s\cdot \eps$, as well. First, since $\mathcal{H}_5^{ds}(F)=\lim_{n\to \infty}\mathcal{H}_{n,5}^{ds}(F)=0$, then there exists $n_0\in \N$ such that $\mathcal{H}_{n,5}^{ds}(F)<\gamma$ for all $n\geq n_0$, where $\gamma=\frac{\eps'}{\kappa\cdot \kappa'}$ with $\kappa$ and $\kappa'$ being the constants provided by both the $\kappa-$condition and the finitely splitting property that stand for $\gf$. Let $n\geq n_0$. Since 
\begin{align*}
\mathcal{H}_{n,5}^{ds}(F)&=\inf\left\{\sum_{i\in I}\dm(A_i)^{ds}:\{A_i\}_{i\in I}\in \mathcal{A}_{\Delta_n}^5(F)\right\},\text{ where }\\
\mathcal{A}_{\Delta_n}^5(F)&=\{\{A_i\}_{i\in I}:\text{ for all } i\in I, \text{ there exists } m\geq n:A_i\in \Delta_m,\\
&F\subseteq \cup_{i\in I}A_i\},
\end{align*}
then there exists $\{A_i\}_{i\in I}$ satisfying the three following:
\begin{enumerate}[(i)]
\item $F\subseteq \cup_{i\in I}A_i$.\label{it:1}
\item For all $i\in I$, there exists $m\geq n$ such that $A_i\in \Delta_m$ with $\dm(A_i)\leq \dm(\Delta_m)\leq\dm(\Delta_{n_0})$, and \label{it:21}
\item $\sum_{i\in I}\dm(A_i)^{ds}<\gamma$.\label{it:3}
\end{enumerate} 
In addition, for all $i\in I$, let $n_i\in \N$ be such that 
\begin{align}\label{eq:diam-ni}
\dm(\Delta_{n_i})\leq \dm(A_i)<\dm(\Delta_{n_i-1}).
\end{align}
By both (\ref{it:21}) and Eq.~(\ref{eq:diam-ni}), it holds that $\dm(\Delta_{n_i})\leq\dm(A_i)\leq\dm(\Delta_{n_0})$. 
Thus, $n_i\geq n_0$ for all $i\in I$. 
Next, we shall define an appropriate covering for the elements in $\{A_i\}_{i\in I}$. Let
\begin{align*}
\mathcal{A}&=\cup_{i\in I}\mathcal{A}_i, \text{ where }\\
\mathcal{A}_i&=\{C\in \Delta_{n_i}:C\cap A_i\neq \emptyset\}\text{ for all } i\in I.
\end{align*}
It is worth pointing out that $\St(A_i,\Delta_{n_i})=\cup\{C\in \Delta_{n_i}:C\cap A_i\neq \emptyset\}=\cup\{C:C\in \mathcal{A}_i\}$.
The four following hold:
\begin{enumerate}[(1)]
\item $\mathcal{A}$ is a covering of $F$. In fact, 
$F\subseteq \cup_{i\in I} A_i\subseteq \cup_{i\in I}\cup_{C\in \mathcal{A}_i}C=\cup_{C\in \mathcal{A}}C$, where the first inclusion is due to (\ref{it:1}) and the second one stands since $A_i\subseteq \cup\{C:C\in \mathcal{A}_i\}$ for each $A_i\in \{A_i\}_{i\in I}$. 
\item $\sum_{A\in \mathcal{A}}\dm(A)^{ds}<\eps'$. Indeed, observe that \label{it:2} 
\begin{align*}
\sum_{A\in \mathcal{A}}\dm(A)^{ds}&=\sum_{i\in I}\sum_{A\in \mathcal{A}_i}\dm(A)^{ds}\leq \sum_{i\in I}\sum_{A\in \mathcal{A}_i}\dm(\Delta_{n_i})^{ds}\\
&\leq \sum_{i\in I} \kappa\cdot \kappa'\cdot \dm(\Delta_{n_i})^{ds}\leq \kappa\cdot \kappa' \sum_{i\in I}\dm(A_i)^{ds}\\
&<\kappa\cdot \kappa'\cdot \gamma=\eps',
\end{align*}
where the first inequality holds since $\dm(A)\leq \dm(\Delta_{n_i})$ for all $A\in \mathcal{A}_i$.
It is worth mentioning that the second inequality stands by applying both the $\kappa-$condition and the finitely splitting property. In fact, for all $i\in I$, it holds that $\dm(A_i)<\dm(\Delta_{n_i-1})$ (c.f.~Eq.~(\ref{eq:diam-ni})), so $A_i$ intersects to $\leq \kappa$ elements in $\Delta_{n_i-1}$ by the $\kappa-$condition. Hence, $A_i$ intersects to $\leq \kappa\cdot \kappa'$ elements in $\Delta_{n_i}$ since $\gf$ is finitely splitting. Thus, 
$\cd(\mathcal{A}_i)\leq \kappa\cdot \kappa'$. 
Eq.~(\ref{eq:diam-ni}) also yields the third inequality. Notice also that (\ref{it:3}) has been applied to deal with the last one.  
\item For all $C\in \mathcal{A}$, there exists $n_i\geq n_0$ such that $C\in \Delta_{n_i}$. Thus, we can write $C=\alpha(C')$ for some $C'\in \Gamma_{n_i}$. By Main hypotheses \ref{hypo:main}, there exist $c\neq 0$ and $d\in \R$ such that $\dm(C)^d=c\cdot \dm(C')$ for all $C\in \mathcal{A}$. Thus, we have
\[
\sum_{C'\in \mathcal{A}'}\dm(C')^s=\frac{1}{c^s}\,\sum_{C\in \mathcal{A}}\dm(C)^{ds}<\frac{\eps'}{c^s}=\eps,
\] 
where $\mathcal{A}'=\cup_{i\in I}\mathcal{A}'_i$ and $\mathcal{A}'_i=\{C'\in \Gamma_{n_i}:\alpha(C')\in \mathcal{A}_i\}$ for all $i\in I$. It is worth noting that (\ref{it:2}) has been applied in the inequality above.

\item $\alpha^{-1}(F)\subseteq \cup_{C'\in \mathcal{A}'}C'$. Let $x\in \alpha^{-1}(F)$. We shall prove that there exists $C'\in \mathcal{A}'$ such that $x\in C'$. First, we have $\alpha(x)\in F$. Since $F\subseteq \cup_{i\in I}A_i$ by (\ref{it:1}), then $\alpha(x)\in A_i$ for some $i\in I$.
On the other hand, let $C'\in \Gamma_{n_i}$ be such that $x\in C'$. Then $C'\in \mathcal{A}'_i$, if and only if, $\alpha(C')\in \mathcal{A}_i$. In this way, observe that $\alpha(C')\in \Delta_{n_i}$ with $n_i\geq n_0$ since $C'\in \Gamma_{n_i}$. Next, we verify that $\alpha(C')\cap A_i\neq \emptyset$. Indeed, $\alpha(x)\in \alpha(C')$ since $x\in C'$. Thus, $\alpha(x)\in \alpha(C')\cap A_i$, so $\alpha(C')\cap A_i\neq \emptyset$. Therefore, $\alpha(C')\in \mathcal{A}_i$ and hence, $C'\in \mathcal{A}'_i$. Accordingly, $x\in C'\in \mathcal{A}'_i\subseteq \mathcal{A}'$. 
\end{enumerate}
The previous calculations allow justifying that for all $\eps>0$, there exists $n_0\in \N$ such that $\mathcal{H}_{n,5}^s(\alpha^{-1}(F))<\eps$ for all $n\geq n_0$. Equivalently, $\mathcal{H}_{5}^s(\alpha^{-1}(F))=0$. 
\end{proof}

\begin{theorem}\label{cor:>}
Let $F\subseteq Y$ and $s\geq 0$. Assume that $\gf$ is finitely splitting and satisfies the $\kappa-$condition. Under Main hypotheses~\ref{hypo:main}, it holds that
\[
d\cdot \dim_{\ef}^5(\alpha^{-1}(F))\leq \dim_{\gf}^5(F).
\]
\end{theorem}

\begin{proof}
In fact, by Proposition \ref{prop:awk}, it holds that $\mathcal{H}_5^{ds}(F)=0$ implies $\mathcal{H}_5^s(\alpha^{-1}(F))=0$. Thus, for all $s> \frac 1d\cdot \dim_{\gf}^5(F)$, we have $s\geq \dim_{\ef}^5(\alpha^{-1}(F))$, and hence the desired equality stands.
\end{proof}
The next key result holds as a consequence of previous Theorems \ref{cor:<} and \ref{cor:>}.
\begin{theorem}\label{teo:dim5-to-R}
Let $F\subseteq Y$. Assume that $\gf$ is finitely splitting and satisfies the $\kappa-$condition. Under Main hypotheses~\ref{hypo:main}, we have
\[
\dim_{\gf}^5(F)=d\cdot \dim_{\ef}^5(\alpha^{-1}(F)).
\]
\end{theorem}
Without too much effort, both Propositions \ref{prop:ds<s} and \ref{prop:awk} as well as Theorems \ref{cor:<}, \ref{cor:>}, and \ref{teo:dim5-to-R} can be proved to stand for fractal dimension IV under the same hypotheses. Thus, we also have the next result for that fractal dimension, which only involves finite coverings and becomes especially appropriate for empirical applications.
\begin{theorem}\label{teo:dim4-to-R}
Let $F\subseteq Y$. Assume that $\gf$ is finitely splitting and satisfies the $\kappa-$condition. Under Main hypotheses~\ref{hypo:main}, it holds that
\[
\dim_{\gf}^4(F)=d\cdot \dim_{\ef}^4(\alpha^{-1}(F)).
\]
\end{theorem}

The following result regarding fractal dimension IV stands similarly to Theorem~\ref{teo:bc-to-R-via-dim3}.
\begin{theorem}\label{teo:H-to-R-via-dim4}
Let $F$ be a compact subset of $[0,1]^d$, $\gf$ be the natural fractal structure on $[0,1]^d$, and $\alpha:X\to [0,1]^d$ a function between the GF-spaces $(X,\ef)$ and $([0,1]^d,\gf)$ with $\gf=\alpha(\ef)$. 
Under Main hypotheses \ref{hypo:main}, Hausdorff dimension of $F$ equals the fractal dimension IV of $\alpha^{-1}(F)$ multiplied by the embedding dimension, $d$, i.e.,
\[
\h(F)=d\cdot \cuatro(\alpha^{-1}(F)).
\] 
\end{theorem}

\begin{proof}
In fact, it is worth noting that
\[
\h(F)=\dim_{\gf}^4(F)=d\cdot\cuatro(\alpha^{-1}(F)),
\]
where the first equality stands by Theorem \ref{teo:previo} (\ref{teo:4=H}) since $\gf$ is the natural fractal structure on $F$ and the last identity is due to Theorem \ref{teo:dim4-to-R}.
\end{proof}
Accordingly, the previous result guarantees that Hausdorff dimension of each compact subset $F$ of $[0,1]^d$ can be calculated in terms of the fractal dimension IV of $\alpha^{-1}(F)\subseteq [0,1]$. Thus, the Algorithm provided in \cite{MFM15} may be applied with this aim.


\section{Calculating both the box and Hausdorff dimensions in higher dimensional spaces}\label{sec:4}

The next remark becomes useful for upcoming purposes.
\begin{remark}\label{obs:dms}
Let $F\subseteq Y$ and $n\in \N$. Under Main hypotheses \ref{hypo:main}, it holds that 
\[
\dm(\alpha^{-1}(F),\Gamma_n)\to 0\Leftrightarrow \dm(F,\Delta_n)\to 0.
\]  
\end{remark}

\begin{proof}~

\begin{enumerate}
\item [($\Rightarrow$)] Let $\eps>0$. Then there exists $n_0\in \N$ such that $\dm(\alpha^{-1}(F),\Gamma_n)<\gamma$ for all $n\geq n_0$ with $\gamma=\frac{\eps^d}{c}$. Thus, $\dm(A)<\gamma$ for all $A\in \Gamma_n$ with $A\cap \alpha^{-1}(F)\neq \emptyset$ and $n\geq n_0$. Hence, Main hypotheses \ref{hypo:main} imply that $\dm(\alpha(A))<(c\cdot \gamma)^{\frac 1d}=\eps$ for all $A\in \Gamma_n$ with $A\cap \alpha^{-1}(F)\neq \emptyset$ and $n\geq n_0$. Accordingly, Lemma \ref{lema:1} leads to $\dm(\alpha(A))<\eps$ for all $A\in \Gamma_n$ with $\alpha(A)\cap F\neq \emptyset$ and $n\geq n_0$, so $\dm(F,\Delta_n)\to 0$.
\item [($\Leftarrow$)] Let $\eps>0$. Then there exists $n_0\in \N$ such that $\dm(\alpha(A))<\gamma$ for all $A\in \Gamma_n$ with $\alpha(A)\cap F\neq \emptyset$ and $n\geq n_0$, where $\gamma=(c\cdot \eps)^{\frac 1d}$. Since $\dm(\alpha(A))=(c\cdot \dm(A))^{\frac 1d}<\gamma$ for all $A\in \Gamma_n$ and $n\in \N$ by Main hypotheses \ref{hypo:main}, then we can affirm that $\dm(A)<\eps$ for all $A\in \Gamma_n$ with $\alpha(A)\cap F\neq \emptyset$. By Lemma \ref{lema:1} we have that $\dm(A)<\eps$ for all $A\in \Gamma_n$ with $A\cap \alpha^{-1}(F)\neq \emptyset$ and $n\geq n_0$, so $\dm(\alpha^{-1}(F),\Gamma_n)\to 0$.
\end{enumerate}
\end{proof}

It is worth pointing out that both results \cite[Theorem 4.1]{GARCIA2017} and \cite[Theorem 1]{SkubalskaRafajowicz2005} allow the calculation of the box dimension of a given subset $F$ in terms of the box dimension of a lower dimensional set connected with $F$ via either a $\delta-$uniform curve or a quasi-inverse function, respectively. However, both of them stand for Euclidean subsets. Next, we provide a similar result in a more general setting.
 
\begin{theorem}\label{teo:bc-to-R}
Assume that Main hypotheses \ref{hypo:main} are satisfied, let $F\subseteq Y$, and assume that $\dm(F,\Delta_n)\to 0$. If both fractal structures $\ef$ and $\gf$ lie under the $\kappa-$condition, then the (lower/upper) box dimension of $F$ equals the (lower/upper) box dimension of $\alpha^{-1}(F)$ multiplied by $d$. In particular, if $\bc(F)$ exists, then $\bc(\alpha^{-1}(F))$ also exists (and reciprocally), and it holds that 
\[
\bc(F)=d\cdot \bc(\alpha^{-1}(F)).
\]
\end{theorem}

\begin{proof}
In fact, the following chain of identities holds for lower/upper dimensions:
\[
\bc(F)=\dim_{\gf}^2(F)=d\cdot \dim_{\ef}^2(\alpha^{-1}(F))=d\cdot \bc(\alpha^{-1}(F)),
\]
where the first and the last equalities hold by both Theorem \ref{teo:k-cond->bc=dos} and Remark \ref{obs:dms}, and the second one is due to Theorem~\ref{teo:dim2-to-R}.
\end{proof}



The next remark regarding the existence of the box dimension of $\alpha^{-1}(F)$ (resp., of $F$) should be highlighted.
\begin{remark}
It is worth pointing out that, under the hypotheses of Theorem \ref{teo:bc-to-R}, 
$\bc(F)$ exists, if and only if, $\bc(\alpha^{-1}(F))$ exists. 
\end{remark}



Next step is to prove a similar result to Theorem \ref{teo:bc-to-R} for Hausdorff dimension. To deal with, first we provide the following

\begin{prop}\label{prop:k-cond->H=cinco}
Let $\ef$ be a finitely splitting fractal structure on $X$ satisfying the $\kappa-$condition with $\dm(\Gamma_n)\to 0$. Then $\mathcal{H}_{\rm H}^s(F)=0$ implies $\mathcal{H}_5^s(F)=0$ for all subset $F$ of $X$.
\end{prop}

\begin{proof}
Let $\eps>0$ and $s\geq 0$ be such that $\mathcal{H}^s_{\rm H}(F)=0$. Since $\mathcal{H}_{\rm H}^s(F)=\lim_{\delta\to 0}\mathcal{H}_{\delta}^s(F)=\sup_{\delta>0}\mathcal{H}_{\delta}^s(F)=0$, then there exists $\delta_0>0$ such that $\mathcal{H}_{\delta}^s(F)<\gamma$ for all $\delta<\delta_0$, where $\gamma=\frac{\eps}{\kappa\cdot \kappa'}$ with $\kappa$ and $\kappa'$ being the constants provided by both the $\kappa-$condition and the finitely splitting property, resp., that stand for $\ef$. In addition, let $n_0\in \N$ be such that $\dm(\Gamma_{n_0})<\delta_0$. Thus, $\mathcal{H}^s_{\dm(\Gamma_{n_0})}(F)<\gamma$. Hence, there exists a family of subsets $\{B_i\}_{i\in I}$ satisfying that
\begin{enumerate}[(i)]
\item $F\subseteq \cup_{i\in I}B_i$.
\item $\dm(B_i)\leq \dm(\Gamma_{n_0})$ for all $i\in I$.\label{it:ii}
\item $\sum_{i\in I}\dm(B_i)^s<\gamma$.\label{it:iii}
\end{enumerate}
For each $i\in I$, let $n_i\in \N$ be such that 
\begin{equation}\label{eq:diameters}
\dm(\Gamma_{n_i})\leq \dm(B_i)<\dm(\Gamma_{n_i-1}).
\end{equation}
Moreover, for each $B_i\in \{B_i\}_{i\in I}$, we shall define a covering by elements in level $n_i$ of $\ef$. In fact, let $\mathcal{A}_i=\{A\in \Gamma_{n_i}:A\cap B_i\neq \emptyset\}$ for all $i\in I$ and $\mathcal{A}=\cup_{i\in I}\mathcal{A}_i$, as well.
Accordingly, the five following hold:
\begin{enumerate}[(1)]
\item $\St(B_i,\Gamma_{n_i})=\cup\{A\in \Gamma_{n_i}:A\cap B_i\neq \emptyset\}=\cup\{A:A\in \mathcal{A}_i\}$ for all $i\in I$. 
\item $n_i\geq n_0$ for all $i\in I$. In fact, notice that $\dm(\Gamma_{n_i})\leq \dm(B_i)\leq \dm(\Gamma_{n_0})$ for all $i\in I$, where the first inequality stands by Eq.~(\ref{eq:diameters}) and the second one is due to (\ref{it:ii}).
\item $\mathcal{A}$ covers $F$. Indeed, $F\subseteq \cup_{i\in I}B_i\subseteq \cup_{i\in I}\cup_{A\in \mathcal{A}_i}A=\cup_{A\in \cup_{i\in I}\mathcal{A}_i}A=\cup_{A\in \mathcal{A}}A$.
\item For all $A\in \mathcal{A}$, there exists $i\in I$ such that $A\in \mathcal{A}_i$, namely, $A\in \Gamma_{n_i}$ with $n_i\geq n_0$ (and $A\cap B_i\neq \emptyset$).
\item $\sum_{A\in \mathcal{A}}\dm(A)^s<\eps$. In fact,
\begin{align*}
\sum_{A\in \mathcal{A}}\dm(A)^s&=\sum_{i\in I}\sum_{A\in \mathcal{A}_i}\dm(A)^s\leq \sum_{i\in I}\sum_{A\in \mathcal{A}_i}\dm(\Gamma_{n_i})^s\\
&\leq \sum_{i\in I}\kappa\cdot \kappa'\cdot \dm(\Gamma_{n_i})^s\leq \kappa\cdot \kappa'\, \sum_{i\in I}\dm(B_i)^s\\
&<\kappa\cdot \kappa'\cdot \gamma=\eps,
\end{align*}
where the first inequality stands since $\dm(A)\leq \dm(\Gamma_{n_i})$ for all $A\in \mathcal{A}_i$. Moreover, the second inequality above holds since $\cd(\mathcal{A}_i)\leq \kappa\cdot \kappa'$ for all $i\in I$. In fact, $\ef$ lies under the $\kappa-$condition, so the number of elements in $\Gamma_{n_i-1}$ that are intersected by each $B_i$ is $\leq \kappa$ with $\dm(B_i)<\dm(\Gamma_{n_i-1})$ (c.f.~Eq.~(\ref{eq:diameters})). Therefore, $B_i$ intersects to $\leq \kappa\cdot \kappa'$ elements in $\Gamma_{n_i}$ by additionally applying the finitely splitting property, also standing for $\ef$. The third one follows since $\dm(\Gamma_{n_i})\leq \dm(B_i)$ for all $i\in I$ (c.f. Eq.~(\ref{eq:diameters})). Finally, we have applied (\ref{it:iii}) to deal with the last inequality. 
\end{enumerate}
Accordingly, the calculations above allow justifying that for all $\eps>0$, there exists $n_0\in \N$ such that $\mathcal{H}_{n,5}^s(F)<\eps$ for all $n\geq n_0$, namely, $\mathcal{H}_5^s(F)=0$.
\end{proof}

\begin{theorem}\label{teo:k-cond->H=cinco}
Let $\ef$ be a finitely splitting fractal structure on $X$ satisfying the $\kappa-$condition with $\dm(\Gamma_n)\to 0$.
Then $\h(F)=\dim_{\ef}^5(F)$.
\end{theorem}

\begin{proof}
First, it is clear that $\h(F)\leq \cinco(F)$ since $\mathcal{A}_{n,5}(F)\subseteq \mathcal{C}_{\delta}(F)$ for all $F\subseteq X$ and $n\in \N$. In fact, each covering in the family $\mathcal{A}_{n,5}(F)$ becomes a $\delta-$cover for an appropriate $\delta>0$. Conversely, let $s\geq 0$. Since $\ef$ is finitely splitting and lies under the $\kappa-$condition, then $\mathcal{H}_{\rm H}^s(F)=0$ implies $\mathcal{H}_5^s(F)=0$ for all subset $F$ of $X$ (c.f. Proposition~\ref{prop:k-cond->H=cinco}). Thus, $s\geq \cinco(F)$ for all $s\geq \h(F)$ and in particular, $\h(F)\geq \cinco(F)$.
\end{proof}

\begin{theorem}\label{teo:H-to-R}
Assume that both fractal structures $\ef$ and $\gf$ are finitely splitting and lie under the $\kappa-$condition with $\dm(\Gamma_n)\to 0$. Under Main hypotheses \ref{hypo:main}, it holds that
$\h(F)=d\cdot \h(\alpha^{-1}(F))$.
\end{theorem}

\begin{proof}
The following chain of identities holds:
\[
\h(F)=\dim_{\gf}^5(F)=d\cdot \dim_{\ef}^5(\alpha^{-1}(F))=d\cdot \h(\alpha^{-1}(F)),
\]
where both the first and the last equalities stand by Theorem \ref{teo:k-cond->H=cinco} and the second identity is due to Theorem \ref{teo:dim5-to-R}.
\end{proof}
It is worth mentioning that Theorem \ref{teo:H-to-R} could be also proved for compact subsets in terms of fractal dimension IV. In fact, it is clear that both Proposition \ref{prop:k-cond->H=cinco} and Theorem \ref{teo:k-cond->H=cinco} also stand regarding the fractal dimension IV of each compact subset. Next, we highlight that last result.

\begin{theorem}\label{teo:k-cond->H=cuatro}
Let $\ef$ be a finitely splitting fractal structure on $X$ satisfying the $\kappa-$condition with $\dm(\Gamma_n)\to 0$.
Then $\h(F)=\cuatro(F)$ for all compact subset $F$ of $X$.
\end{theorem}

\section{Results in the Euclidean setting}\label{sec:5}

Along this section, we shall pose more operational versions for both Theorems \ref{teo:bc-to-R} and \ref{teo:H-to-R} in the Euclidean setting to tackle with applications of fractal dimension in higher dimensional spaces. The proof regarding the next theorem follows immediately by applying those results.
\begin{theorem}\label{teo:eu-dim-to-R}
Let $\alpha:X\to Y$ be a function between a pair of GF-spaces, $(X,\ef)$ and $(Y,\gf)$, where $X=[0,1]$ and $Y=[0,1]^d$, with $\gf=\alpha(\ef)$. Assume that both fractal structures $\ef$ and $\gf$ lie under the $\kappa-$condition and suppose that there exist real numbers $c\neq 0$ and $d$ for which the next identity stands for all $A\in \Gamma_n$ and all $n\in \N$ (c.f.~Main hypotheses \ref{hypo:main}):
\[
\dm(\alpha(A))^d=c\cdot \dm(A).
\]
Suppose also that  $\dm(\Gamma_n)\to 0$. The two following hold for all $F\subseteq [0,1]^d$:
\begin{enumerate}[(i)]
\item 
\[
\bc(F)=d\cdot \bc(\alpha^{-1}(F)).
\]\label{teo:bc-to-r}
\item In addition, if both $\ef$ and $\gf$ are finitely splitting, then 
\[
\h(F)=d\cdot \h(\alpha^{-1}(F)).
\]
\end{enumerate}
\end{theorem}


\begin{remark}
As a consequence from Theorem \ref{teo:eu-dim-to-R} (i), the (lower/upper) box dimension of $F\subseteq [0,1]^d$ can be calculated throughout the following (lower/upper) limit:
\[
\bc(F)=d\cdot \lim_{\delta\to 0}\frac{\log \mathcal{N}_{\delta}(\alpha^{-1}(F))}{-\log \delta},
\]
where $\mathcal{N}_{\delta}(\alpha^{-1}(F))$ can be calculated as the number of $\delta=2^{-n}-$cubes that intersect $\alpha^{-1}(F)$ (among other equivalent quantities, c.f. \cite[Equivalent definitions 2.1]{FAL14}).
\end{remark}

The next remark highlights why it could be assumed, without loss of generality, that $F$ is contained in $[0,1]^d$ for box/Hausdorff dimension calculation purposes.
\begin{remark}
Let $F$ be a bounded subset of $\R^d$.
Since the box/Hausdorff dimension is invariant by bi-Lipschitz transformations (c.f.~\cite[Corollary 2.4 (b)/Section 3.2]{FAL14}), an appropriate similarity $f:\R^d\to \R^d$ may be applied to $F$ so that $f(F)\subseteq [0,1]^d$ with $\dim(F)=\dim(f(F))$, where $\dim$ refers to box/Hausdorff dimension.
\end{remark}

Interestingly, it holds that a natural choice for both fractal structures $\ef$ and $\gf$ may be carried out so that they satisfy both the $\kappa-$condition and the finitely splitting property. As such, Theorem \ref{teo:eu-dim-to-R} can be applied to calculate the box/Hausdorff dimension of a subset $F$ of $[0,1]^d$. 
\begin{remark}\label{obs:eu-fdims}
Notice that Theorem \ref{teo:eu-dim-to-R} can be applied in the setting of both GF-spaces $(Y=[0,1]^d,\gf)$ and $(X=[0,1],\ef)$, where $\gf$ can be chosen to be the natural fractal structure on $[0,1]^d$, i.e., $\gf=\{\Delta_n:n\in \N\}$ with levels given by $\Delta_n=\{[\tfrac{k_1}{2^n},\tfrac{1+k_1}{2^n}]\times\ldots\times [\tfrac{k_d}{2^n},\tfrac{1+k_d}{2^n}]:k_1,\ldots,k_d=0,1,\ldots,2^n-1\}$ and $\ef=\{\Gamma_n:n\in \N\}$ with $\Gamma_n=\{[\tfrac{k}{2^{nd}},\tfrac{1+k}{2^{nd}}]:k=0,1,\ldots,2^{nd}-1\}$ for all $n\in \N$. Thus, $\gf$ satisfies both the $\kappa-$condition for $\kappa=3^d$ and the finitely splitting property for $\kappa'=2^d$. In addition, it holds that $\ef$ also lies under both the $\kappa-$condition (for $\kappa=3$) and the finitely splitting property (for $\kappa'=2^d$), as well. Observe that level $n$ of each fractal structure contains $2^{nd}$ elements. Regarding Main hypotheses \ref{hypo:main}, it is worth noting that for such fractal structures there exist $d$ and $c\neq 0$ such that $\dm(\alpha(A))^d=c\cdot \dm(A)$ for all $A\in \Gamma_n$ and all $n\in \N$. In fact, just observe that $\dm(A)=2^{-nd}$ for each $A\in \Gamma_n$. In addition, it holds that $\dm(\alpha(A))=2^{-n}\, \sqrt{d}$. Thus, for $k=d^{\frac d2}$, where $d$ is the embedding dimension, we have $(\frac{\sqrt{d}}{2^n})^d=k\cdot (\frac{1}{2^d})^n$ for all $n\in \N$.  

Following the constructive approach theoretically described in upcoming Theorem \ref{teo:curves}, a function $\alpha:X\to Y$ can be constructed 
so that $\gf=\alpha(\ef)$, and hence, it holds that $\dim(F)=d\cdot \dim(\alpha^{-1}(F))$ for all $F\subseteq Y$, where $\dim$ refers to box/Hausdorff dimension.
\end{remark}

\section{How to construct $\alpha$}\label{sec:how-to-construct}\label{sec:6}

Along this paper, we have been focused on calculating the fractal dimension of a subset $F\subseteq Y$ in terms of the fractal dimension of its pre-image $\alpha^{-1}(F)\subseteq X$ via a function $\alpha:X\to Y$ with $\gf=\alpha(\ef)$ (c.f. Theorems \ref{teo:dim1-to-R}, \ref{teo:dim2-to-R}, \ref{teo:dim3-to-R}, \ref{teo:dim5-to-R}, \ref{teo:dim4-to-R}, \ref{teo:bc-to-R}, \ref{teo:H-to-R}, and \ref{teo:eu-dim-to-R}). In this section, we state a powerful result (c.f. Theorem \ref{teo:curves}) allowing the explicit construction of such a function. To deal with, first let us recall the concepts of Cantor complete fractal structure and starbase fractal structure, as well.

First, it is worth mentioning that a sequence $\{A_n:n\in \N\}$ is decreasing provided that $A_{n+1}\subseteq A_n$ for all $n\in \N$. 
\begin{definition}[\cite{completeGF}, Definition 3.1.1]\label{def:Cantor-complete}
Let $\ef=\{\Gamma_n:n\in \N\}$ be a fractal structure on $X$. We shall understand that $\ef$ is Cantor complete if for each decreasing sequence $\{A_n:n\in \N\}$ with $A_n\in \Gamma_n$, it holds that $\cap_{n\in \N}A_n\neq \emptyset$.
\end{definition}

The concept of a starbase fractal structure also plays a key role to deal with the construction of such a function $\alpha$.
\begin{definition}[\cite{directedGFspaces}, Section 2.2]
Let $\ef$ be a fractal structure on $X$. We say that $\ef$ is starbase if $\St(x,\ef)=\{\St(x,\Gamma_n):n\in \N\}$ is a neighborhood base of $x$ for all $x\in X$.
\end{definition}

The main result in this section is stated next.
\begin{theorem}[\cite{curves}, Theorem 3.6]\label{teo:curves}
Let $\ef=\{\Gamma_n:n\in \N\}$ be a starbase fractal structure on a metric space $X$ and $\gf=\{\Delta_n:n\in \N\}$ be a Cantor complete starbase fractal structure on a complete metric space $Y$. Moreover, let $\{\alpha_n:n\in \N\}$ be a family of functions, where each $\alpha_n:\Gamma_n\to \Delta_n$ satisfies the two following:
\begin{enumerate}
\item [(i)] if $A\cap B\neq \emptyset$ with $A,B\in \Gamma_n$ for some $n\in \N$, then $\alpha_n(A)\cap \alpha_n(B)\neq \emptyset$.\label{cond:i}
\item [(ii)] If $A\subseteq B$ with $A\in \Gamma_{n+1}$ and $B\in \Gamma_n$ for some $n\in \N$, then $\alpha_{n+1}(A)\subseteq \alpha_n(B)$.\label{cond:ii}
\end{enumerate}
Then there exists a unique continuous function $\alpha:X\to Y$ such that $\alpha(A)\subseteq \alpha_n(A)$ for all $A\in \Gamma_n$ and all $n\in \N$. Additionally, if $\ef$ is Cantor complete and each $\alpha_n$ also satisfies the two following:
\begin{enumerate}
\item [(iii)] $\alpha_n$ is onto.\label{cond:iii}
\item [(iv)] $\alpha_n(A)=\cup\{\alpha_{n+1}(B):B\in \Gamma_{n+1},B\subseteq A\}$ for all $A\in \Gamma_n$,\label{cond:iv}
\end{enumerate}
then $\alpha$ is onto and $\alpha(A)=\alpha_n(A)$ for all $A\in \Gamma_n$ and all $n\in \N$.
\end{theorem} 
To properly construct a function $\alpha:X\to Y$ according to Theorem \ref{teo:curves}, we can proceed as follows. First, for each $x\in X$, there exists a decreasing sequence $\{A_n:n\in \N\}$  such that $A_n\in \Gamma_n$ for all $n\in \N$ with $x\in \cap_{n\in \N}A_n$. Thus, $\{\alpha_n(A_n):n\in \N\}$ is also decreasing with $\alpha_n(A_n)\in \Delta_n$ for all $n\in \N$. Further, it holds that $\cap_n \alpha_n(A_n)$ is a single point since $\gf$ is starbase and Cantor complete. Therefore, we shall define $f(x)=\cap_{n\in \N}\alpha_n(A_n)$.

Next, we illustrate how Theorem \ref{teo:curves} allows the construction of functions for Theorem \ref{teo:eu-dim-to-R} application purposes. In this way, let us show how the classical Hilbert's square-filling curve can be iteratively described by levels.
\begin{example}[c.f.~Example 1 in \cite{curves}]\label{example:hilbert}
Let $(Y,\gf)$ be a GF-space with $Y=[0,1]\times [0,1]$ and $\gf$ being the natural fractal structure on $[0,1]\times [0,1]$ as a Euclidean subset, i.e., $\gf=\{\Delta_n:n\in \N\}$, where $\Delta_n=\left\{[\frac{k_1}{2^n},\frac{1+k_1}{2^n}]\times [\frac{k_2}{2^n},\frac{1+k_2}{2^n}]:k_1,k_2=0,1,\ldots,2^n-1\right\}$ for each $n\in \N$. In addition, let $(X,\ef)$ be another GF-space where $X=[0,1]$ and $\ef=\{\Gamma_n:n\in \N\}$ with $\Gamma_n=\left\{[\frac{k}{2^{2n}},\frac{1+k}{2^{2n}}]:k=0,1,\ldots,2^{2n}-1\right\}$. It is worth pointing out that each level $n$ of $\ef$ (resp., of $\gf$) contains $2^{2n}$ elements. Next, we explain how to construct a function $\alpha:X\to Y$ such that $\gf=\alpha(\ef)$. To deal with, we shall define the image of each element in level $n$ of $\ef$ through a function $\alpha_n:\Gamma_n\to \Delta_n$. For instance, let $\alpha([0,\tfrac 14])=[0,\tfrac 12]^2, \alpha([\tfrac 14,\tfrac 12])=[\tfrac 12,1]\times [0,\tfrac 12], \alpha([\tfrac 12,\tfrac 34])=[\tfrac 12,1]^2$, and $\alpha([\tfrac 34,1])=[0,\tfrac 12]\times [\tfrac 12,1]$, as well (c.f. Fig.~\ref{fig:hilbert2}). Thus, the whole level $\Delta_1=\alpha(\Gamma_1)=\{\alpha(A):A\in \Gamma_1\}$ has been defined. It is worth mentioning that this approach provides additional information regarding $\alpha$ as deeper levels of both $\ef$ and $\gf$ are reached via $\alpha_n$ under the two following conditions (c.f.~Theorem~\ref{teo:curves}):
\begin{enumerate}
\item [(i)] if $A\cap B\neq \emptyset$ with $A,B\in \Gamma_n$ for some $n\in \N$, then $\alpha_n(A)\cap \alpha_n(B)\neq \emptyset$.
\item [(ii)] If $A\subseteq B$ with $A\in \Gamma_{n+1}$ and $B\in \Gamma_n$ for some $n\in \N$, then $\alpha_{n+1}(A)\subseteq \alpha_n(B)$.
\end{enumerate}
For instance, if $A\in \Gamma_n$, then we can calculate its image via $\alpha_n$, $\alpha_n(A)$. Going beyond, let $B\in \Gamma_{n+1}$ be so that $B\subseteq A$. Then $\alpha_{n+1}(B)\subseteq \alpha_n(A)$ refines the definition of $\alpha_n(A)$, and so on. This allows us to think of the Hilbert's curve as the limit of the maps $\alpha_n$ (c.f.~Fig.~\ref{fig:hilbert}). This example illustrates how Theorem \ref{teo:curves} allows the construction of (continuous) functions and in particular, space-filling curves.
\begin{figure}
\begin{center}
\begin{tabular}{cc}
\includegraphics[width=4cm,height=5.2cm]{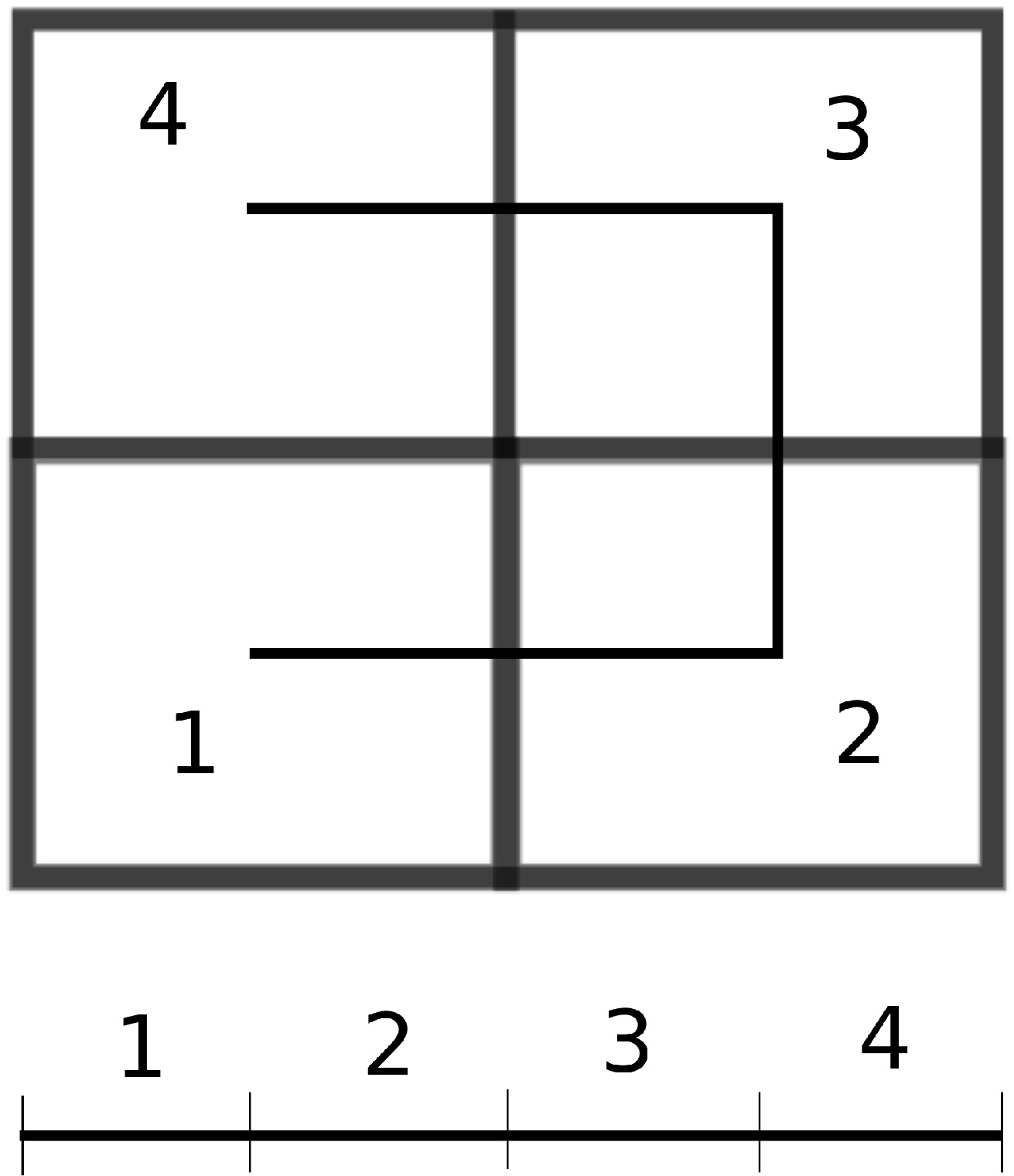} &
\includegraphics[width=4cm,height=7.2cm]{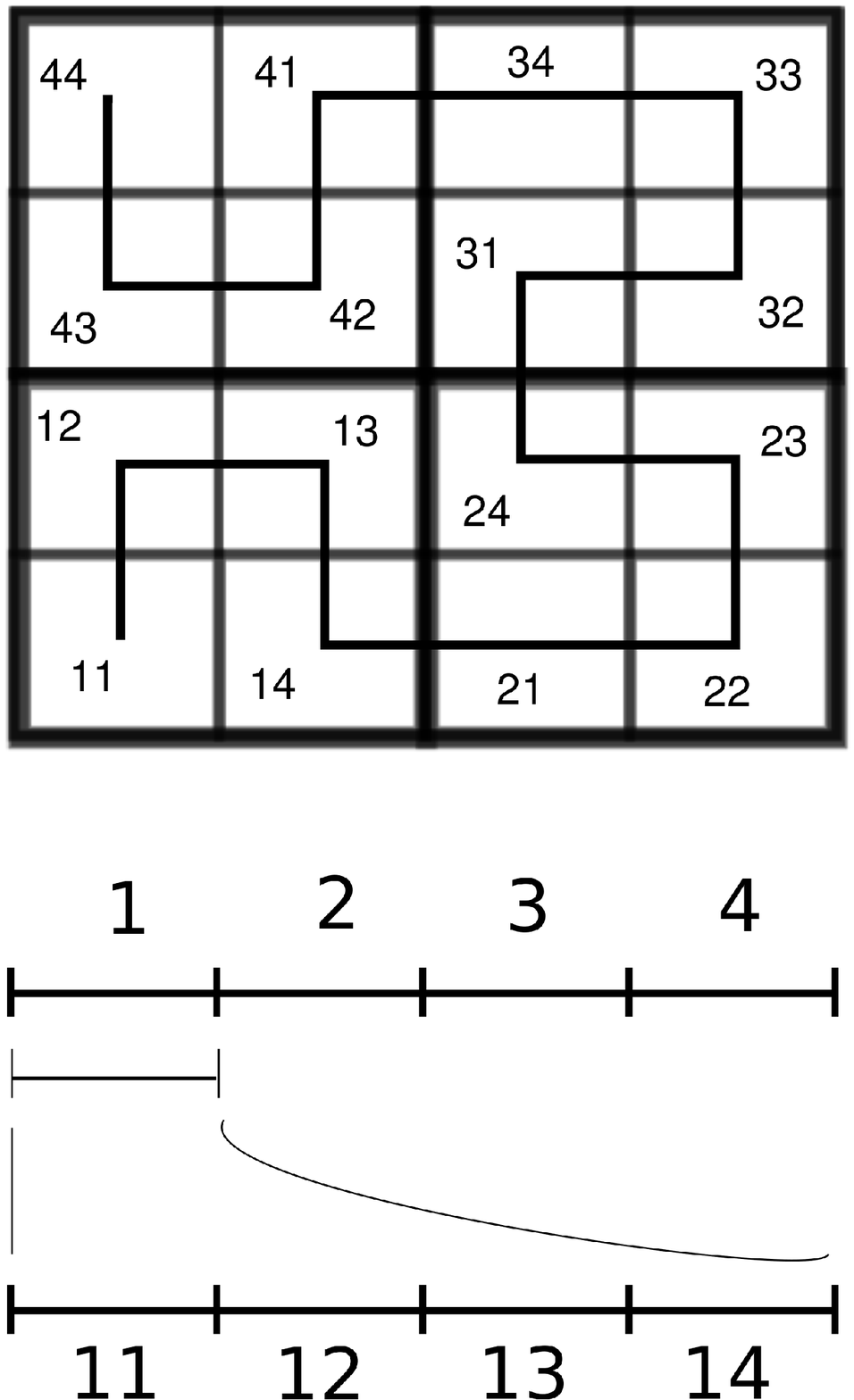}
\end{tabular}
\caption{The two plots above arrange how each element in level $n$ of $\ef$ can be sent to some element in level $n$ of $\gf$ via $\alpha_n$ for $n=2,3$.}\label{fig:hilbert2}
\end{center}
\end{figure}

\begin{figure}
\begin{center}
\begin{tabular}{ccc}
\includegraphics[width=4cm,height=4cm]{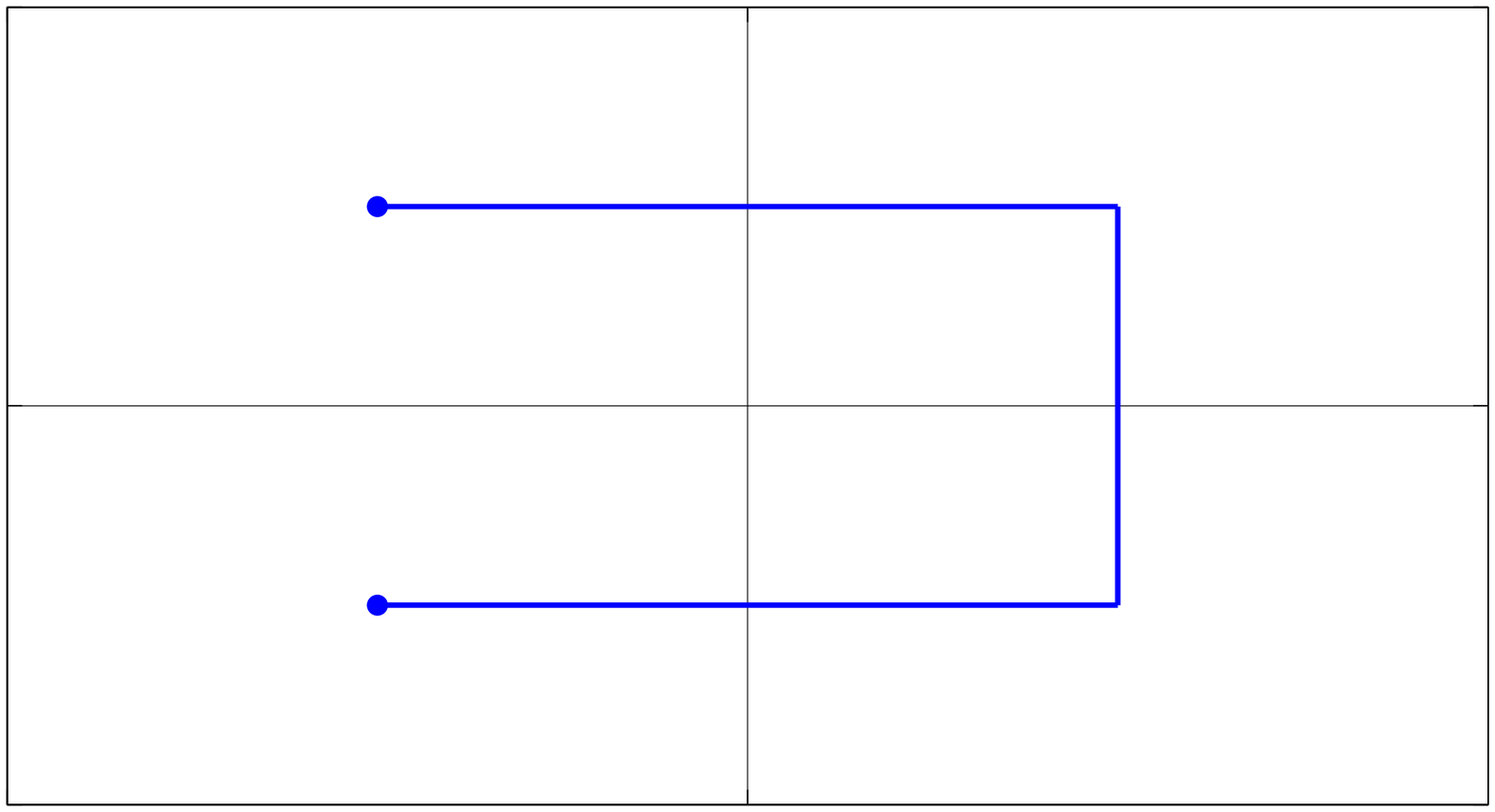} &
\includegraphics[width=4cm,height=4cm]{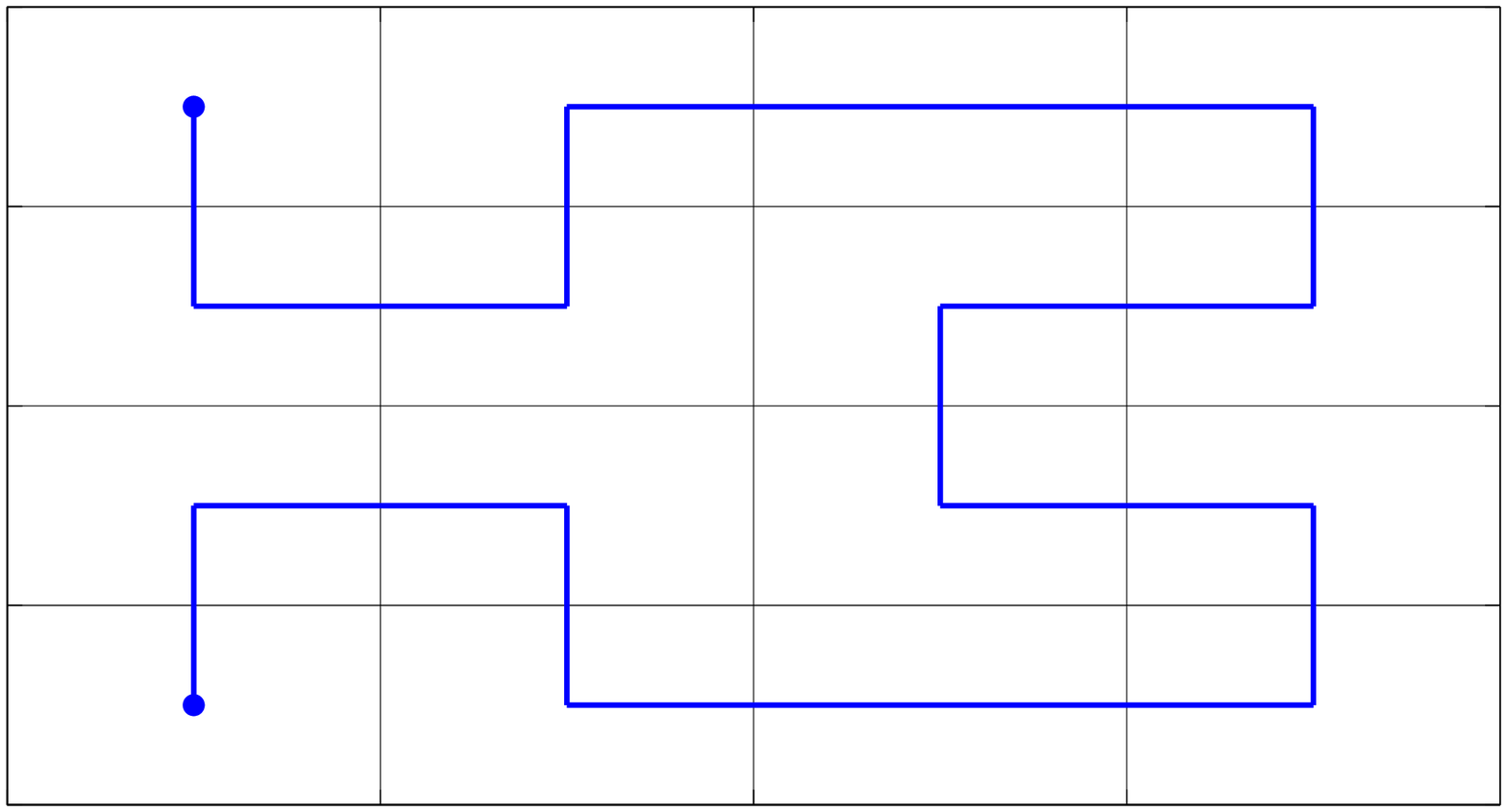} &
\includegraphics[width=4cm,height=4cm]{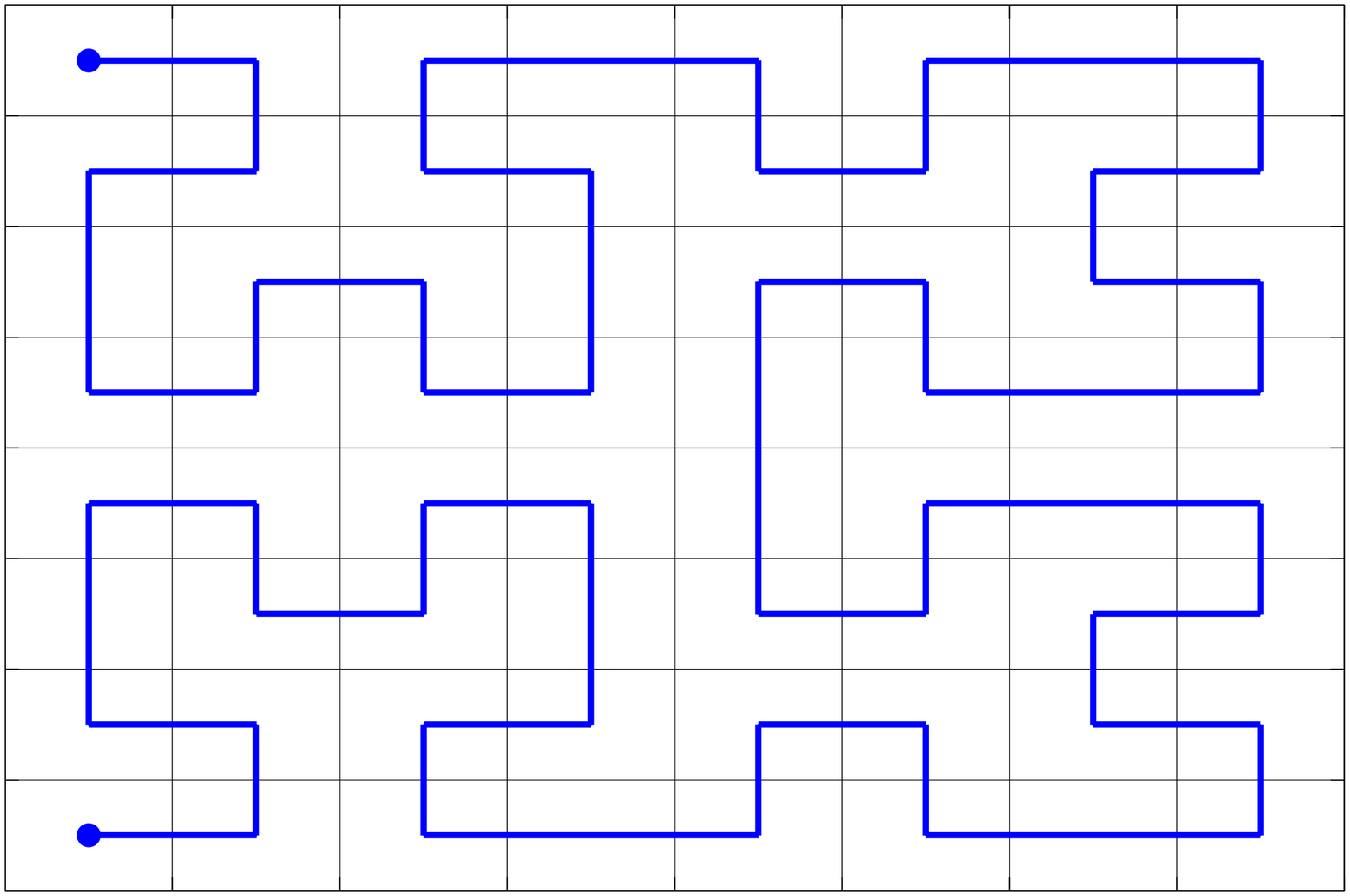}
\end{tabular}
\caption{First three levels in the construction of the Hilbert's curve according to Theorem \ref{teo:curves} (c.f.~\cite[Fig.~2]{curves}).}\label{fig:hilbert}
\end{center}
\end{figure}
\end{example}
It is worth mentioning that Theorem \ref{teo:curves} also allows the construction of maps filling a whole attractor. For instance, in \cite[Example 4]{curves}, we generated a curve crossing once each element of the natural fractal structure which the Sierpi\'nski gasket can be naturally endowed with as a self-similar set (c.f.~Fig.~\ref{fig:sierp-curve}).

\begin{figure}
\begin{center}
\begin{tabular}{cc}
\includegraphics[width=5.7cm,height=4.6cm]{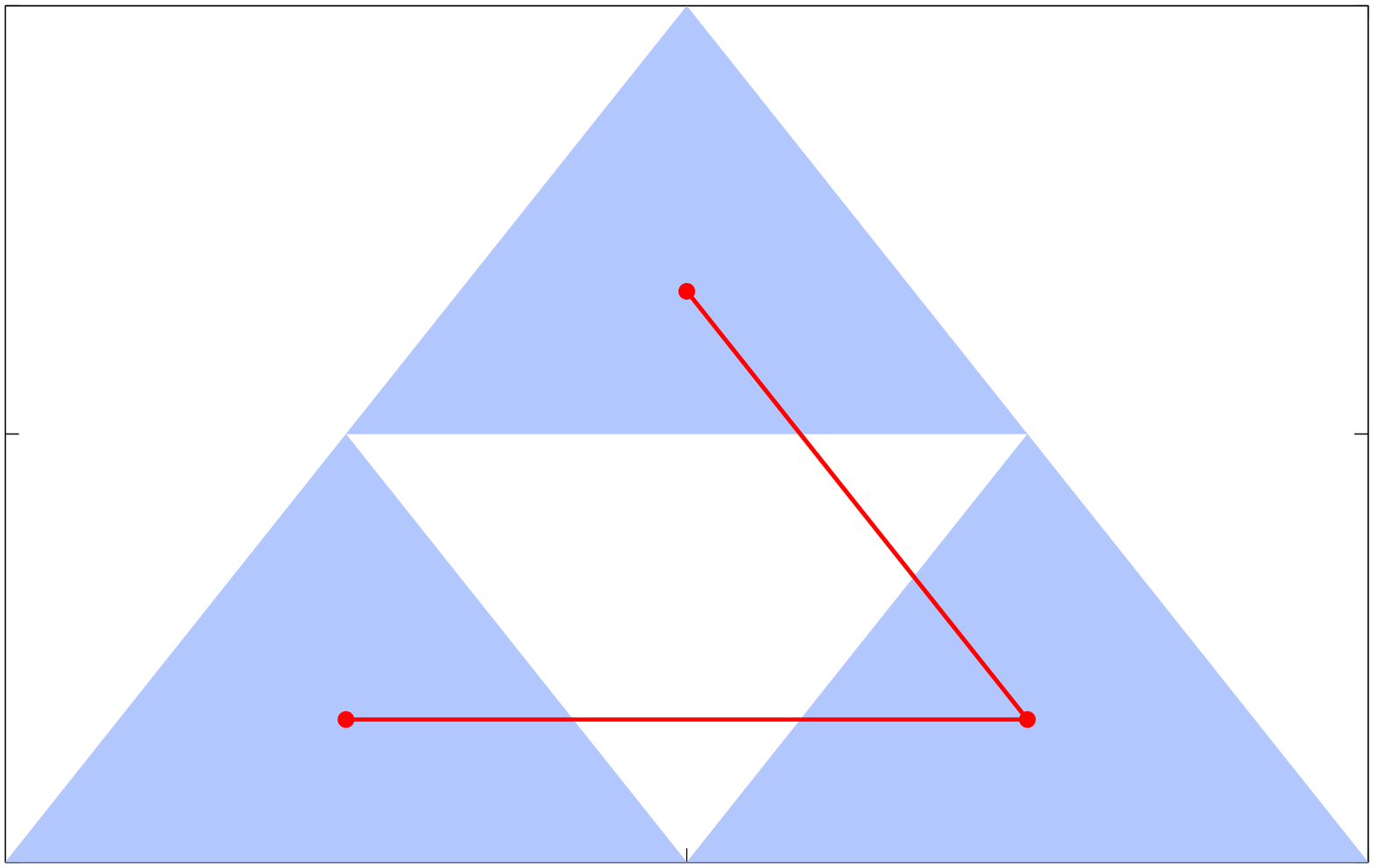} &
\includegraphics[width=5.7cm,height=4.6cm]{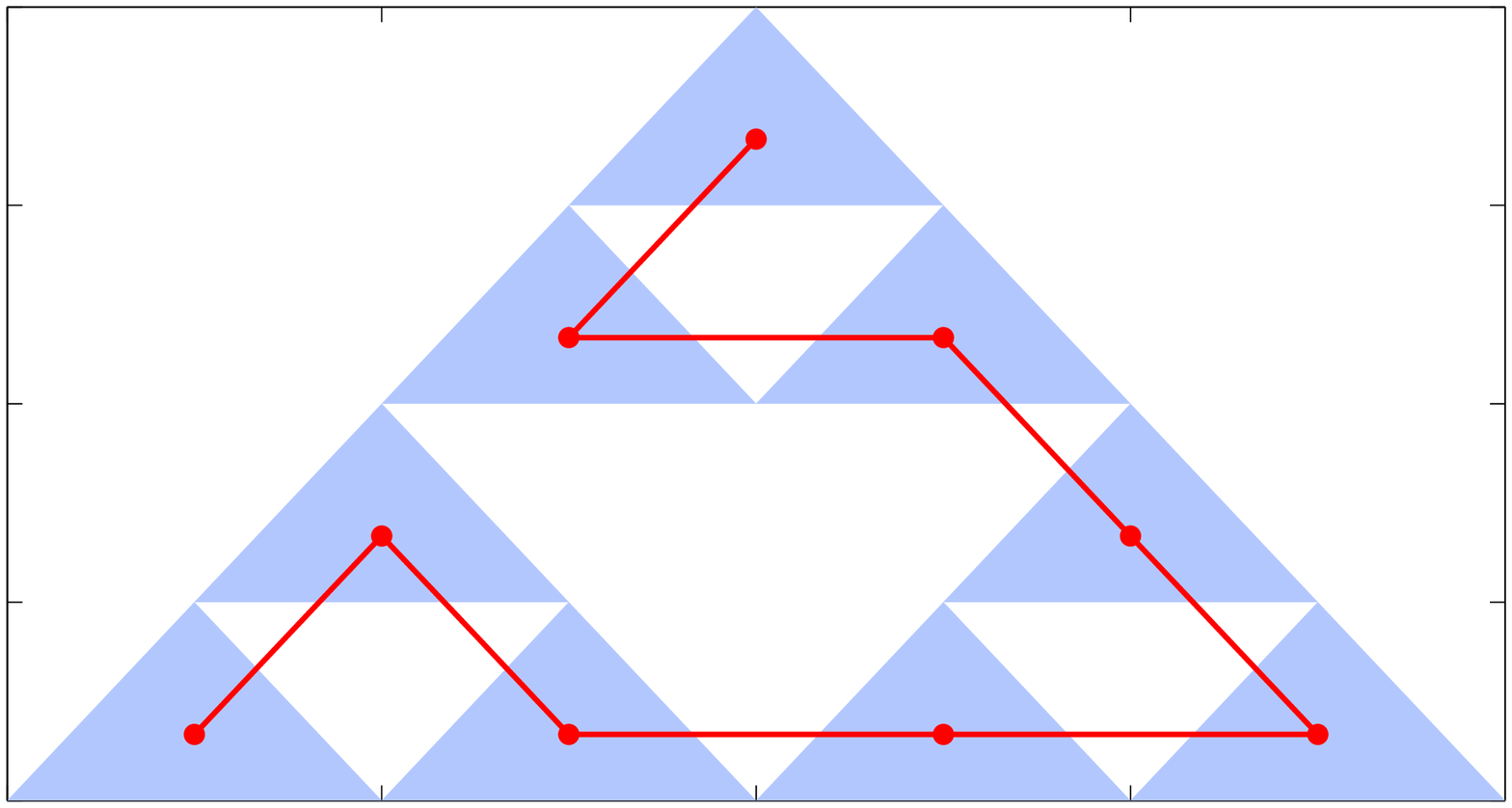}
\end{tabular}
\caption{First two levels in the construction of a curve filling the whole Sierpi\'nski gasket (c.f.~\cite[Fig.~4]{curves}).}\label{fig:sierp-curve}
\end{center}
\end{figure}


\end{document}